\documentclass[hidelinks,onefignum,onetabnum]{siamart220329}

\usepackage{import}
\usepackage[noend]{algpseudocode} 
\usepackage{amsmath,amsfonts,amssymb} 
\usepackage[english]{babel} 
\usepackage{bm}
\usepackage{booktabs}    
\usepackage{hyperref,nameref,cleveref}
\usepackage{comment} 
\usepackage{empheq}
\usepackage[inline]{enumitem}
\usepackage{fancyhdr}
\usepackage[T1]{fontenc}    
\usepackage{gensymb}
\usepackage{graphicx}
\usepackage[utf8]{inputenc} 
\usepackage{latexsym}
\usepackage{lipsum}
\usepackage{listings} 
\usepackage{mathtools}
\usepackage{microtype}
\usepackage{multicol,multirow}
\usepackage[square,sort,comma,numbers]{natbib}
\usepackage{nicefrac}  
\usepackage[parfill]{parskip}
\usepackage{setspace} 
\usepackage{subcaption}
\usepackage{times}
\usepackage{titlesec} 
\usepackage{thmtools}
\usepackage{upquote}
\usepackage{url}            
\usepackage{wrapfig}
\usepackage{xcolor}
\usepackage{xspace}
\usepackage{tikz}
\usepackage{adjustbox}
\usepackage{array}
\usepackage[normalem]{ulem}
\ifpdf
  \DeclareGraphicsExtensions{.eps,.pdf,.png,.jpg}
\else
  \DeclareGraphicsExtensions{.eps}
\fi

\setlist[itemize]{leftmargin=*}
\setlist[enumerate]{leftmargin=*}

\newtheorem{conjecture}[theorem]{Conjecture}

\newtheorem{remark}{Remark}[section]
\newtheorem{example}{Example}[section]

\newsiamremark{hypothesis}{Hypothesis}
\crefname{hypothesis}{Hypothesis}{Hypotheses}

\headers{Robust Blockwise Random Pivoting}{Y. Dong, C. Chen, P. G. Martinsson, K. Pearce}

\title{Robust Blockwise Random Pivoting: \\ 
Fast and Accurate Adaptive Interpolative Decomposition
\thanks{
    Submitted to the editors on \today.
    \funding{
        YD was supported by NYU Courant Instructorship. CC was supported in part by startup funds of North Carolina State University. KP was supported in part by the Peter O'Donnell Jr. Postdoctoral Fellowship at the Oden Institute.
        Overall, this research was supported by the Office of Naval Research (N00014-18-1-2354), the National Science Foundation (DMS-2313434 and DMS-1952735), and the Department of Energy ASCR (DE-SC0022251).
        }
    }
}

\author{
    Yijun Dong\thanks{
        Courant Institute, New York University (\texttt{yd1319@nyu.edu}).
    }
    \and
    Chao Chen\thanks{
        Department of Mathematics, North Carolina State University (\texttt{cchen49@ncsu.edu}).
    }
    \and
    Per-Gunnar Martinsson\thanks{
        Department of Mathematics \& Oden Institute, University of Texas at Austin (\texttt{pgm@oden.utexas.edu}).
    }
    \and
    Katherine Pearce\thanks{
        Oden Institute, University of Texas at Austin (\texttt{katherine.pearce@austin.utexas.edu}).
    }
}

\usepackage{amsopn}

\newcommand{\E}{\mathbb{E}}

\newcommand{\N}{\mathbb{N}} 
\newcommand{\R}{\mathbb{R}} 
\newcommand{\SSS}{\mathbb{S}}

\newcommand{\ab}{\mathbf{a}}

\newcommand{\db}{\mathbf{d}}
\newcommand{\eb}{\mathbf{e}}

\newcommand{\pb}{\mathbf{p}}

\newcommand{\ub}{\mathbf{u}}
\newcommand{\vb}{\mathbf{v}}

\newcommand{\xb}{\mathbf{x}}
\newcommand{\yb}{\mathbf{y}}

\newcommand{\Ib}{\mathbf{I}}

\newcommand{\Lb}{\mathbf{L}}

\newcommand{\Qb}{\mathbf{Q}}
\newcommand{\Rb}{\mathbf{R}}

\newcommand{\Ub}{\mathbf{U}}
\newcommand{\Vb}{\mathbf{V}}
\newcommand{\Wb}{\mathbf{W}}
\newcommand{\Xb}{\mathbf{X}}
\newcommand{\Yb}{\mathbf{Y}}

\newcommand{\Ccal}{\mathcal{C}}

\newcommand{\Ecal}{\mathcal{E}}

\newcommand{\Ncal}{\mathcal{N}}

\newcommand{\Scal}{{\mathcal{S}}}
\newcommand{\Tcal}{{\mathcal{T}}}

\newcommand{\Xcal}{\mathcal{X}}

\newcommand{\Sfrak}{\mathfrak{S}}

\newcommand{\eps}{\epsilon}

\newcommand{\thetab}{\boldsymbol{\theta}}

\newcommand{\mub}{\boldsymbol{\mu}}

\newcommand{\xib}{\boldsymbol{\xi}}

\newcommand{\pib}{\boldsymbol{\pi}}

\newcommand{\Sigmab}{\boldsymbol{\Sigma}}

\newcommand{\Omegab}{\boldsymbol{\Omega}}

\renewcommand{\le}{\leqslant}
\renewcommand{\ge}{\geqslant}
\renewcommand{\leq}{\leqslant}

\newcommand{\dfeq}{\triangleq}

\newcommand{\pinv}{\dagger}

\newcommand*{\argmin}{\mathop{\mathrm{argmin}}}
\newcommand*{\argmax}{\mathop{\mathrm{argmax}}}

\newcommand{\diag}{\mathop{\mathrm{diag}}}
\newcommand{\rank}{\mathop{\mathrm{rank}}}

\newcommand{\nucnorm}[1]{{\left\vert\kern-0.25ex\left\vert\kern-0.25ex\left\vert #1 
    \right\vert\kern-0.25ex\right\vert\kern-0.25ex\right\vert}}

\newcommand{\wh}[1]{\widehat{#1}}

\newcommand{\eg}{\emph{e.g.}\xspace}
\newcommand{\ie}{\emph{i.e.}\xspace}
\newcommand{\iid}{\emph{i.i.d.}\xspace}
\newcommand{\cf}{\emph{cf.}\xspace}

\newcommand{\rbr}[1]{\left(#1\right)}
\newcommand{\sbr}[1]{\left[#1\right]}
\newcommand{\cbr}[1]{\left\{#1\right\}}
\newcommand{\nbr}[1]{\left\|#1\right\|}
\newcommand{\abbr}[1]{\left\vert#1\right\vert}
\newcommand{\abr}[1]{\langle#1\rangle}
\newcommand{\nnbr}[1]{{\left\vert\kern-0.25ex\left\vert\kern-0.25ex\left\vert #1 \right\vert\kern-0.25ex\right\vert\kern-0.25ex\right\vert}}

\newcommand{\csep}[2]{\left\{#1 \middle\vert #2 \right\}}

\newcommand{\csepp}[2]{\left\{#1 ~\middle\vert~ #2 \right\}}

\newcommand{\bmat}[1]{\begin{bmatrix} #1 \end{bmatrix}}

\definecolor{commentcolor}{RGB}{110,154,155}   

\renewcommand{\b}{\textbf}
\renewcommand{\t}[1]{\text{#1}}

\newcommand{\tsvd}[2]{{\abr{#1}}_{#2}}
\newcommand{\errskel}[2]{\Ecal_{#1}\rbr{#2}}
\newcommand{\errid}[3]{\Ecal_{#1}\rbr{#3 \mid #2}}

\newcolumntype{H}{>{\setbox0=\hbox\bgroup}c<{\egroup}@{}}

\ifpdf
\hypersetup{
  pdftitle={Robust Blockwise Random Pivoting},
  pdfauthor={Yijun Dong, Chao Chen, Per-Gunnar Martinsson, and Katherine Pearce}
}
\fi

\begin{document}

\maketitle

\begin{abstract}
    The interpolative decomposition (ID) aims to construct a low-rank approximation formed by a basis consisting of row/column skeletons in the original matrix and a corresponding interpolation matrix. 
    This work explores fast and accurate ID algorithms from comprehensive perspectives for empirical performance,
    {including accuracy in both skeleton selection and interpolation matrix construction, efficiency in terms of asymptotic complexity and hardware efficiency, as well as rank adaptiveness.}
    While many algorithms have been developed to optimize some of these aspects, practical ID algorithms proficient in all aspects remain absent. 
    To fill in the gap, we introduce \emph{robust blockwise random pivoting (RBRP)} that is 
    {asymptotically fast, hardware-efficient, and rank-adaptive, providing accurate skeletons and interpolation matrices comparable to the best existing ID algorithms in practice.}
    Through extensive numerical experiments on various synthetic and natural datasets, we demonstrate the appealing empirical performance of RBRP from the aforementioned perspectives, as well as the robustness of RBRP to adversarial inputs.
\end{abstract}

\begin{keywords}
Randomized Numerical Linear Algebra, Interpolative Decomposition, Column Subset Selection, Adaptive Sampling
\end{keywords}

\begin{MSCcodes}
65F55, 65Y05, 68T05
\end{MSCcodes}

\section{Introduction}

The interpolative decomposition (ID) is a special type of low-rank approximation where a basis for the row/column space of a given target matrix is formed by subselecting rows/columns. 
This allows an ID approximation to identify structure and essential information within data. 
{The ID has found a wide range of applications in numerical analysis, data compression, and machine learning. 
For example, in fast direct solvers for linear elliptic partial differential equations (PDEs), the ID is a widely used method for compressing certain off-diagonal blocks of rank-structured matrices~\citep{martinsson2005fast,gillman2012direct,ho2012fast,ho2013hierarchical,corona2015n,minden2017recursive}.
Another emerging application of the ID lies in the low-rank adaptation (LoRA)~\citep{hu2021lora} for parameter-efficient fine-tuning of large language models~\citep{houlsby2019parameter,lester2021power} where ID of weight matrices can be leveraged as good initialization for the low-rank adapter matrices~\citep{wang2024pmss}.
More broadly, the idea of ID naturally connects to various related problems like CUR decomposition, kernel quadrature, optimal experimental design, data selection, active learning, and model pruning~\citep{mahoney2009cur,epperly2023kernel,pukelsheim2006optimal,avron2013faster,dasgupta2008hierarchical,dong2024randomly,axiotis2024data,li2024greedy}.}

Given $\Xb = \sbr{\xb_1, \cdots, \xb_n}^T \in \R^{n \times d}$ consisting of $n$ data points $\cbr{\xb_i \in \R^d}_{i \in [n]}$, along with a constant $\eps > 0$ and a target rank $r \in \N$, we aim to construct {a $(r,\eps)$-ID of $\Xb$}:
\begin{align}\label{eq:id}
    \underset{n \times d}{\Xb} \approx \underset{n \times k}{\Wb}\ \underset{k \times d}{\Xb_S}
    \quad \emph{s.t.} \quad
    \nbr{\Xb - \Wb\Xb_S}_F^2 \le (1+\eps) \nbr{\Xb - \tsvd{\Xb}{r}}_F^2
\end{align}
where {$k$ is the (unknown) target rank;
$\tsvd{\Xb}{r}$ is the optimal rank-$r$ approximation of $\Xb$;
$S = \cbr{s_1,\cdots,s_k} \subset [n]$ are the skeleton indices corresponding to the \emph{row skeleton} submatrix $\Xb_S = [\xb_{s_1},\cdots,\xb_{s_k} ]^\top \in \R^{k \times d}$; and}
$\Wb \in \R^{n \times k}$ is 
{an interpolation matrix that (nearly) minimizes the error} $\nbr{\Xb - \Wb\Xb_S}_F^2$ for the given $\Xb_S$.
{While we focus on row-wise ID throughout this work, the column-wise ID can be defined and conducted analogously by considering the transpose of $\Xb$.}

The construction of an ID can be divided into two stages: 
\begin{enumerate}[label=(\roman*)]
    \item Select a skeleton subset $S \subset [n]$ that achieves a small \emph{skeletonization error}
    \begin{align}\label{eq:skeletonization_error}
        \errskel{\Xb}{S} \dfeq \nbr{\Xb  - \Xb \Xb_S^\dagger \Xb_S}_F^2 = \min_{\Wb \in \R^{n \times \abbr{S}}} \nbr{\Xb  - \Wb \Xb_S}_F^2.
    \end{align}
    \item Given $S$, compute an interpolation matrix $\Wb \in \R^{n \times \abbr{S}}$ efficiently\footnote{
        Although \eqref{eq:skeletonization_error} directly provides $\Wb = \Xb\Xb_S^\dagger$ as the minimizer of $\min_{\Wb}\errid{\Xb}{S}{\Wb} = \errskel{\Xb}{S}$, the explicit computation of $\Xb\Xb_S^\dagger$ takes $O(ndk)$ time (\eg~via \eqref{eq:stable_id}), which is undesired. Therefore, whenever available, information from the skeleton selection process is usually leveraged to evaluate/approximate $\Wb$ efficiently in $o(ndk)$ time (as elaborated in \Cref{sec:id_construction}), which may lead to the gap $\errid{\Xb}{S}{\Wb} > \errskel{\Xb}{S}$.
    } that (approximately) minimizes the \emph{interpolation error}: $\errid{\Xb}{S}{\Wb} \dfeq \nbr{\Xb  - \Wb \Xb_S}_F^2$. 
\end{enumerate}
In this work, we explore fast and accurate ID algorithms from both the skeleton subset selection and interpolation matrix construction perspectives, starting with the question:
\begin{center}
    \emph{How do we construct a $(r,\eps)$-ID efficiently with as small a skeleton $S$ as possible, \\
    without prior knowledge of the target rank $k$?}
\end{center}

\subsection{What are Fast and Accurate ID Algorithms?}
To formalize the concept of ``fast and accurate'' ID algorithms, we dissect the main question above into the following properties, together providing systematic performance measurements for ID algorithms from both the efficiency and accuracy perspectives.
\begin{enumerate}[label=(\alph*)]
    \item {\b{Skeleton complexity} measures the utility of skeletons selected by an ID algorithm in terms of the minimum possible rank of a $(r,\eps)$-ID for any given $\Xb$. Precisely, the skeleton complexity is a property of an ID algorithm defined as the minimum rank $k$ such that for any $\Xb$, $\abbr{S}=k$ is sufficient to guarantee either a $(r,\eps)$-ID for a deterministic algorithm or $\E[\errskel{\Xb}{S}] \le (1+\eps) \nbr{\Xb - \tsvd{\Xb}{r}}_F^2$ for a randomized algorithm.}
    
    \item \b{Asymptotic complexity} measures the number of floating point operations (FLOPs) in the skeleton selection process.
    
    \item {\b{Hardware efficiency}} refers to the implementation advantage that dominant cost of skeleton selection in an ID algorithm appears as matrix-matrix, instead of matrix-vector, multiplications. 
    On modern processors, Level-3 BLAS, or matrix-matrix, operations can leverage memory cashing much more efficiently than Level-2 or Level-1 BLAS operations~\citep{blackford2002updated,goto2008high}. As a result, a sequence of matrix operations implemented using Level-3 BLAS can attain significantly higher performance than an optimal implementation using Level-2 or Level-1 BLAS, especially on parallel processors.
    
    \item {\b{Rank-adaptiveness}} refers to the ability to evaluate the skeletonization error efficiently on the fly {without requiring prior knowledge of the target rank $k$.
    \begin{remark}\label{def:error_revealing}
        An ID algorithm is rank-adaptive if after selecting a skeleton subset $S$, it can evaluate $\errskel{\Xb}{S}$ efficiently with at most $O\rbr{n}$ operations.
        That is, given any relative error tolerance $\tau > 0$, {a rank-adaptive} ID algorithm can determine an appropriate rank $k$ with negligible (\ie no more than $O(nk)$) additional cost, while selecting a skeleton subset $S$ of size $k$ such that $\errskel{\Xb}{S} \le \tau \nbr{\Xb}_F^2$. 
    \end{remark}}
    For the sake of analysis, we define a useful constant associated with $\Xb$:
    \begin{align}\label{eq:relative_tail}
        \eta_r \dfeq {\nbr{\Xb - \tsvd{\Xb}{r}}_F^2}/{\nbr{\Xb}_F^2},
    \end{align}
    which quantifies the relative tail weight of $\Xb$ with respect to the rank $r$ and represents the relative optimal rank-$r$ approximation error of $\Xb$.
    Notice that when $\tau < (1+\eps)\eta_r$, the algorithm outputs a $(r,\eps)$-ID.
    
    \item \b{ID-revealing property} characterizes if the skeleton selection stage of an ID algorithm contains sufficient information for efficient construction of the interpolation matrix.
    {\begin{definition}[ID-revealing properties]\label{def:id_revealing}
        We say a skeleton selection algorithm is
        \begin{itemize}
            \item \b{ID-revealing} if it contains sufficient information in addition to $S$ such that the optimal interpolation matrix $\Wb = \Xb\Xb_S^\dagger$ can be evaluated efficiently in $O(nk^2)$ time;
            \item \b{$\gamma$-ID-revealing} for some $\gamma > 1$ if it can construct a suboptimal interpolation matrix $\Wb \approx \Xb\Xb_S^\dagger$ with $\errid{\Xb}{S}{\Wb} \le \gamma \errskel{\Xb}{S}$ in $O(nk^2)$ time; and
            \item \b{non-ID-revealing} if neither the optimal interpolation matrix nor its approximations can be constructed in $O(nk^2)$ time, \ie, the interpolation matrix can only be constructed by solving the least-squares problem \eqref{eq:skeletonization_error} in $O(ndk)$ time.
        \end{itemize}
    \end{definition}}
\end{enumerate}

\subsection{How to Combine Adaptiveness and Randomness for ID?}
Adaptiveness and randomness are two key algorithmic properties {for skeleton selection}. 
On the one hand, with adaptiveness, skeleton selection in each step is aware of {the previous selections} so that redundant skeleton selection can be better avoided. 
On the other hand, randomness {balances the exploitation with exploration}, thereby improving the robustness to adversarial inputs. 

To ground the notions of adaptiveness and randomness, we synopsize some representative existing skeleton selection algorithms below, as well as in \Cref{tab:adaptiveness_randomness_summary}, with a focus on the aforementioned properties for performance measurement.
\begin{enumerate}[label=(\alph*)]
    \item \b{Greedy pivoting} is a classical strategy in numerical linear algebra that involves \emph{only adaptiveness}. As an example of greedy pivoting, column-pivoted QR (CPQR)~\citep[Section 5.4.2]{golub2013matrix} {starts with the input matrix $\Xb$ as the active submatrix,} picks the row with the maximum norm in each step, and adaptively updates the {active submatrix} via projection onto the orthogonal complement of the skeletons. 
    {CPQR is rank-adaptive and ID-revealing, and usually provides low skeleton complexities in practice.}
    However, deterministic greedy pivoting algorithms like CPQR are inherently sequential and vulnerable to adversarial inputs where the skeleton complexity can approach the problem size $n$~\cite{kahan1966numerical,chen2022randomly}.
    
    \item \b{Sampling} comprises the well-researched body of skeleton selection methods that involve \emph{only randomness}. Some common examples related to ID include squared-norm sampling~\citep{frieze2004fast}, leverage score sampling~\citep{mahoney2009cur}, and DPP sampling~\citep{belabbas2009spectral,kulesza2011k,derezinski2021determinantal}. As a trade-off between the skeleton complexity and efficiency, the fast ($O(nd)$-time) sampling methods like squared-norm sampling tend to suffer from high skeleton complexities that depend heavily on the matrix~\citep{deshpande2006matrix,chen2022randomly}; whereas constructions of the more sophisticated distributions like leverage score and DPP are usually expensive~\citep{drineas2012fast,hough2006determinantal,derezinski2019fast}. Moreover, for ID, sampling methods generally fail to be {rank-adaptive} or ID-revealing.
    
    \item \b{Random pivoting} combines adaptiveness and randomness by replacing the greedy selection of maximum norm in CPQR with random sampling according to the squared-norm distribution associated with the residual~\citep{deshpande2006matrix,deshpande2006adaptive,chen2022randomly}. Closely related to the ID problem considered in this work, the idea of random pivoting has been revitalized by the inspiring recent work \cite{chen2022randomly} in the context of column Nystr\"om approximation with a nearly optimal skeleton complexity guarantee~{\citep[Theorem 5.1]{chen2022randomly}}. However, analogous to CPQR, although random pivoting enjoys the desired {rank-adaptiveness} and ID-revealing properties, the sequential nature of random pivoting compromises its efficiency in practice.
    
    \item \b{Sketchy pivoting} {combines adaptiveness and randomness alternatively by first sketching $\Xb$ from right and then selecting skeletons through greedy pivoting on the sketched matrix~\citep{voronin2017efficient,dong2021simpler}.}
    In contrast to random pivoting, the cost of sketchy pivoting is dominated by the embarrassingly parallelizable sketching process. As a trade-off for superior empirical efficiency, sketchy pivoting sacrifices the {rank-adaptiveness} since sketching requires prior knowledge of the target $k$ or its overestimation.
    Furthermore, sketchy pivoting is not ID-revealing due to the loss of information during sketching. 
    As a simple but effective remedy for the loss of accuracy, we observe that multiplicative oversampling can remarkably improve the gap between the interpolation and skeletonization error $\errid{\Xb}{S}{\Wb} - \errskel{\Xb}{S}$, which we refer to as \emph{oversampled sketchy ID (OSID)}. 
    In \Cref{subsec:inexact_id_revealing_algo}, we show that with sufficiently large oversampling, OSID is $\gamma$-ID-revealing for some small constant $\gamma > 1$.
\end{enumerate}

\begin{table}[!h]
    \centering
    \vspace{-1em}
    \caption{
        Summary of performance for skeleton selection methods that leverage adaptiveness and/or randomness.
        Recall from \eqref{eq:relative_tail} that $\eta_r \dfeq {\nbr{\Xb - \tsvd{\Xb}{r}}_F^2}/{\nbr{\Xb}_F^2}$.
        For randomized methods, ($\E$) denotes that $\E[\errskel{\Xb}{S}] \le (1+\eps) \nbr{\Xb - \tsvd{\Xb}{r}}_F^2$ holds in expectation.
        The skeleton complexities in $(\cdot)^*$ are conjectured based on extensive empirical evidence without formal proofs.
        In the ``Dominant Cost'' column (showing both the asymptotic complexity and {hardware efficiency}), ``m-v'' stands for matrix-vector multiplications, which are significantly less efficient than ``m-m''---matrix-matrix multiplications~\cite{goto2008high}.
    }\label{tab:adaptiveness_randomness_summary}
    \begin{adjustbox}{width=1\textwidth}
    \begin{tabular}{c|cccc}
    \toprule
    Skeleton Selection & Skeleton Complexity & Dominant Cost & {Rank-adaptive} & ID-revealing \\
    \midrule
    \shortstack{Greedy pivoting/CPQR \\ (\citep[Section 5.4.2]{golub2013matrix})} & $k \ge \rbr{1-\rbr{1+\eps}\eta_r}n$ & $O(ndk)$ m-v & \b{Yes} & \b{Yes} \\
    \midrule
    \shortstack{Squared-norm sampling \\ (\eqref{eq:squared_norm_distribution}~\citep{frieze2004fast})} & ($\E$) $k \ge \frac{r-1}{\eps \eta_r} + \frac{1}{\eps}$ & $O(nd)$ \b{m-m} & No & No \\
    \midrule
    \shortstack{Random pivoting \\ (\Cref{algo:pqr_id}~\citep{chen2022randomly,deshpande2006matrix,deshpande2006adaptive})} & \shortstack[c]{($\E$) $k \ge k_\t{RP} = \frac{r}{\eps} + $ \\ $r \min\cbr{\log\rbr{\frac{1}{\eps \eta_r}}, 1 + \log\rbr{\frac{2^r}{\eps}}}$} & $O(ndk)$ m-v & \b{Yes} & \b{Yes} \\
    \midrule
    \shortstack{Sketchy pivoting \\ (\cite[Algorithm 1]{dong2021simpler}\citep{voronin2017efficient})} & ($\E$) $\rbr{k \gtrsim k_\t{RP}}^*$ & $O(ndk)$ \b{m-m} & No & No \\
    \midrule
    \shortstack{\b{RBRP} (this paper) \\ (\Cref{algo:bg_pqr_id})} & ($\E$) $\rbr{k \gtrsim k_\t{RP}}^*$ & $O(ndk)$ \b{m-m} & \b{Yes} & \b{Yes} \\
    \bottomrule
    \end{tabular}
    \end{adjustbox}
\end{table}

To fill in a critical missing piece in the family of existing methods in \Cref{tab:adaptiveness_randomness_summary} that leverage adaptiveness and randomness, in this work, we introduce \b{robust blockwise random pivoting (RBRP, \Cref{algo:bg_pqr_id})}---an ID algorithm that is \emph{$O(ndk)$-time, hardware-efficient, {rank-adaptive}, and ID-revealing}, with comparable skeleton complexities to the nearly optimal random pivoting~\citep{chen2022randomly} in practice.

In particular, the plain blockwise random pivoting (\eg an extension of the blocked RPCholesky algorithm introduced in {\cite[Algorithm 3]{chen2022randomly}}) tends to suffer from unnecessarily large skeleton complexity (\cf \Cref{fig:id_gmm_rpgp} (left)) under adversarial inputs (\eg~\Cref{ex:gmm_pitfall_plain_bas}) due to the lack of local adaptiveness within each block. 
As a remedy, RBRP leverages \emph{robust blockwise filtering} (\Cref{rmk:robust_blockwise_filtering})---applying CPQR to every small sampled block locally and discarding the potentially redundant points through a truncation on the relative residual of the CPQR. 
By choosing a reasonable block size, such robust blockwise filtering effectively resolves the inefficiency in skeleton complexity encountered by the plain blockwise random pivoting (\cf \Cref{fig:id_gmm_rpgp} (right)), with negligible additional cost.

\subsection{Notations and Roadmap}
For $\Xb \in \R^{n \times d}$, let $\Xb = \Ub \Sigmab \Vb^\top = \sum_{i=1}^{\min(n,d)} \sigma_i \ub_i \vb_i^\top$ be the singular value decomposition of $\Xb$ with $\sigma_1 \ge \cdots \ge \sigma_{\min(n,d)} \ge 0$. Given any $r \in [\min(n,d)]$, we denote $\Ub_r = \sbr{\ub_1,\cdots,\ub_r}$, $\Sigmab_r = \diag\rbr{\sigma_1,\cdots,\sigma_r}$, $\Vb_r = \sbr{\vb_1,\cdots,\vb_r}$ such that $\tsvd{\Xb}{r} = \Ub_r \Sigmab_r \Vb_r^\top$.
We follow the MATLAB notation for vector and matrix indexing throughout this work. 
For any $n \in \N$, let $[n] = \cbr{1,\cdots,n}$ and $\Delta_n \in [0,1]^n$ be the probability simplex of dimension $n$.
Also, let $\Sfrak_n$ be the set of all permutations of $[n]$.
For any distribution $\pb \in \Delta_n$ and $k \in \N$, $\pb^k$ represents the joint distribution over $\Delta_n^k = \Delta_n \times \cdots \times \Delta_n$ such that $S = \cbr{s_j \sim \pb}_{j \in [k]} \sim \pb^k$ is a set of $k$ $\iid$ samples.
For any $n \in \N$, let $\eb_i$ ($i \in [n]$) be the $i$-th canonical base of $\R^n$.
For any $m \in \N$, let $\b{1}_{m} \in \R^m$ and $\b{0}_{m} \in \R^m$ be the vector with all entries equal to one and zero, respectively.
We adapt the standard asymptotic notations: for any functions $f,g: \R_+ \to \R_+$, we write $f = O\rbr{g}$ or $f \lesssim g$ if there exists some constant $C>0$ such that $f(x) \leq C g(x)$ for all $x \in \R_+$; $f = \Omega\rbr{g}$ or $f \gtrsim g$ if $g = O\rbr{f}$; $f = \Theta(g)$ or $f \asymp g$ if $f = O\rbr{g}$ and $f = \Omega\rbr{g}$.

As a roadmap, we review the adaptiveness-only and randomness-only skeleton selection algorithms in \Cref{sec:adaptive_or_random}, along with the two different combinations of adaptiveness and randomness in \Cref{sec:adaptive_and_random}. Then, we introduce the RBRP algorithm formally in \Cref{sec:rbas}. In \Cref{sec:id_construction}, we discuss the efficient construction of interpolation matrices after skeleton selection. In \Cref{sec:experiments}, we demonstrate the appealing empirical accuracy and efficiency of RBRP via extensive numerical experiments. The code, as well as a high-performance implementation of RBRP, is available at \href{https://github.com/dyjdongyijun/Robust_Blockwise_Random_Pivoting}{https://github.com/dyjdongyijun/Robust\_Blockwise\_Random\_Pivoting}.

\section{Adaptiveness vs. Randomness}\label{sec:adaptive_or_random}

\subsection{Adaptiveness via Greedy Squared-norm Pivoting}
Greedy pivoting is a classical way of incorporating adaptiveness in skeleton selection~\citep{gu1996efficient,sorensen2016deim,voronin2017efficient}.
Column pivoted QR (CPQR)~\citep[Section 5.4.2]{golub2013matrix} is one of the most commonly used greedy pivoting methods. 
The pivoting strategy and adaptive update are two key components that characterize a greedy pivoting method. In particular, given an active submatrix $\Xb^{(t)} \in \R^{n \times d}$ with at most $n-t$ nonzero rows, CPQR involves greedy squared-norm pivoting and adaptive QR updates\footnote{
    {Notice that \eqref{eq:qr_update} can be viewed as the $(t+1)$th-step of either CPQR on $\Xb^\top$ or row pivoted LQ (RPLQ) on $\Xb$. Throughout this work, we refer to column pivoting on $\Xb^\top$ as CPQR without ambiguity.}
    In \eqref{eq:qr_update}, we consider the exact arithmetic and express the QR update as a Gram-Schmidt process for illustration purposes. In practice, QR updates are usually conducted via Householder reflection~\citep[Section 5.1.2]{golub2013matrix} for numerical stability.
}:
\begin{align}\label{eq:qr_update}
    s_{t+1} \gets \argmax_{i \in [n]} \nbr{\Xb^{(t)}\rbr{i,:}}_2^2, \ 
    \Xb^{(t+1)} \gets \Xb^{(t)} - \Xb^{(t)} \frac{\Xb^{(t)}(s_{t+1},:)^\top \Xb^{(t)}(s_{t+1},:)}{\nbr{\Xb^{(t)}(s_{t+1},:)}_2^2}.
\end{align}
The rank-revealing property of CPQR~\citep[Theorem 7.2]{gu1996efficient} implies that the first $k$ pivots $S=\cbr{s_1,\cdots,s_k}$ form a reasonable skeleton subset with $\errskel{\Xb}{S} \le 4^{k} \rbr{n-k} \nbr{\Xb - \tsvd{\Xb}{k}}_F^2$, where the exponential dependence on $k$ is tight under adversarial inputs (\eg Kahan matrices~\citep{kahan1966numerical}\citep[Example 1]{gu1996efficient}). 
In the worst case, {\citep[Theorem C.2]{chen2022randomly}} shows that CPQR requires almost all the $n$ points ($k \ge \rbr{1 - (1+\eps)\eta_r}n$) before finding a $(r,\eps)$-ID.

\begin{remark}[Greedy squared-norm pivoting is vulnerable but empirically successful]\label{rmk:greedy_squared_norm_pivot}
    The vulnerability to adversaries is a common and critical drawback of squared-norm pivoting in CPQR (and greedy pivoting methods in general).
    Nevertheless, it has been widely observed and conjectured that such adversarial inputs are extremely scarce in practice~\citep{trefethen1990average}. Therefore, CPQR serves as a good skeleton selection method for most data matrices empirically~\citep{voronin2017efficient,dong2021simpler}, as long as orthonormality is maintained to high precision \cite{2005_martinsson_skel}.
\end{remark}

\begin{remark}[CPQR is inherently sequential but {rank-adaptive}]\label{rmk:cpqr_update}
    From the efficiency perspective, a major drawback of CPQR is that the adaptive updates are inherently sequential ({only half of the operations can be cast into high-performance matrix-matrix multiplication in the best-known algorithm that underlies LAPACK's routine geqp3~\cite{quintana1998blas}}). As a result, CPQR tends to be observably slower than simple alternatives with the same asymptotic complexity like Gaussian elimination with partial pivoting~\citep[Section 3.4]{golub2013matrix}---another common greedy pivoting method that retains parallelizability at a cost of the rank-revealing guarantee.

    Nevertheless, the adaptive nature grants CPQR an appealing bonus---the {rank-adaptiveness}. 
    {It is known that the skeletonization error of CPQR, \ie $\errskel{\Xb}{S} = \nbr{\Xb^{\rbr{\abbr{S}}}}_F^2$, can be downdated efficiently at the end of each update.}
\end{remark}

\subsection{Randomness via Squared-norm Sampling}
Sampling is a widely used approach for random skeleton selection (also referred to as column subset selection, \eg in \cite{derezinski2020improved,cortinovis2020low}), where the intuitive goal is to construct a distribution $\pb \in \Delta_n$ over the $n$ data points $\cbr{\xb_i}_{i \in [n]}$ according to their relative ``importance.''
Analogous to the squared-norm pivoting in CPQR (\eqref{eq:qr_update}), sampling according to the squared-norm distribution:
\begin{align}\label{eq:squared_norm_distribution}
    p_i = {\nbr{\xb_i}_2^2}/{\nbr{\Xb}_F^2} \quad \forall~ i \in [n]
\end{align}
is a natural choice for such ``importance'' sampling.
In particular, \cite{deshpande2006matrix} and {\cite[Theorem C.3]{chen2022randomly}} shows that squared-norm sampling requires $k \ge \frac{r-1}{\eps \eta_r} + \frac{1}{\eps}$ skeleton points to form a $(r,\eps)$-ID in expectation.
Compared to the skeleton complexity of pivoting (CPQR), the squared-norm sampling provides 
{a lower skeleton complexity in expectation that scales proportionally to $r$ instead of $n$.}

Despite the improved expected skeleton complexity compared to CPQR, like CPQR, squared-norm sampling can also be vulnerable to adversarial inputs by selecting redundant skeletons with higher probabilities. Such potential inefficiency can be reflected by the linear dependence of the skeleton complexity $k \ge \frac{r-1}{\eps \eta_r} + \frac{1}{\eps}$ on $\frac{r}{\eta_r}$, as instantiated in \Cref{ex:adversary_squared_norm_sampling}.
\begin{example}[Adversarial input for squared-norm sampling]\label{ex:adversary_squared_norm_sampling}
    For $n \in \N$ divisible by $k \in \N$, $r \in \N$ such that $k = (1+\beta) r < n$ for some $\beta \in \N$, letting $\alpha_i = \sqrt{\alpha} \gg \beta \ge 1$ for all $1 \le i \le r$, and $\alpha_i = 1$ for all $r+1 \le i \le k$, we consider the following data matrix 
    \begin{align*}
        \Xb^\top = \sbr{\alpha_1 \eb_1 \b{1}_{n/k}^\top, \cdots, \alpha_r \eb_r \b{1}_{n/k}^\top, \alpha_{r+1} \eb_{r+1} \b{1}_{n/k}^\top, \cdots, \alpha_k \eb_k \b{1}_{n/k}^\top} \in \R^{d \times n}
    \end{align*}
    with SVD 
    \begin{align*}
        \Xb = \underbrace{\bmat{\sqrt{\frac{k}{n}} \b{1}_{n/k} & \cdots & \b{0}_{n/k} \\ \vdots & \ddots & \vdots \\ \b{0}_{n/k} & \cdots & \sqrt{\frac{k}{n}} \b{1}_{n/k}}}_{\Ub \in \R^{n \times k}}
        \underbrace{\bmat{\sqrt{\frac{n}{k}} \alpha_1 && \\ & \ddots & \\ && \sqrt{\frac{n}{k}} \alpha_k}}_{\Sigmab \in \R^{k \times k}} 
        \underbrace{\bmat{\eb_1^\top \\ \vdots \\ \eb_k^\top}}_{\Vb^\top \in \R^{k \times d}}.
    \end{align*}

    Observe that a skeleton subset of size $k$ with the $k$ scaled canonical bases $\csep{\alpha_i \eb_i}{i \in [k]}$ is sufficient to form an ID that exactly recovers $\Xb$. That is, the uniform distribution over the $k$ scaled canonical bases is preferred.

    However, with $\alpha \gg 1$, the squared-norm distribution is heavily skewed to those rows in $\Xb$ pointing toward $\csepp{\eb_i}{1 \le i \le r}$, which leads to redundant samples along those directions:
    \begin{align*}
        p_i = 
        \begin{cases}
            \frac{\alpha}{\alpha r + k -r} = \frac{\alpha}{\rbr{\alpha + \beta} r} \quad &\forall~ 1 \le i \le r \\
            \frac{1}{\alpha r + k -r} = \frac{1}{\rbr{\alpha + \beta} r} \quad &\forall~ r+1 \le i \le k
        \end{cases}.
    \end{align*}  

    Such inefficiency is mirrored in the skeleton complexity $k \ge \frac{r-1}{\eps \eta_r} + \frac{1}{\eps}$ with a small $\eta_r$:
    \begin{align*}
        \eta_r = \frac{k-r}{\alpha r + k - r} = \frac{\beta}{\alpha+\beta} \ll 1.
    \end{align*}
    For example, with $\beta = 1$ and $\alpha \gg k$, the squared-norm sampling suffers from a skeleton complexity $\frac{k}{\eps} + \frac{\alpha (r-1) - r}{\eps} \gg k$ far greater than necessary.

    It is worth highlighting that adaptiveness effectively remedies this inefficiency of squared-norm sampling. For instance, after picking $\alpha_i \eb_i$ as the first skeleton, the adaptive update (\eg \eqref{eq:qr_update}) eliminates the remaining $\frac{n}{k}-1$ rows of $\Xb$ in the $\eb_i$ direction and therefore helps exclude the possible redundant samples. 
\end{example}

Determinantal point process (DPP)~\citep{belabbas2009spectral,derezinski2021determinantal} and leverage score sampling~\citep{mahoney2009cur} are well-studied alternatives to squared-norm sampling that provide better skeleton complexities but with much higher cost of constructing the corresponding distributions as a trade-off. 
For example, the k-DPP sampling~\citep{kulesza2011k} is known for its nearly optimal skeleton complexity: as few as $k \ge \frac{r}{\eps} + r - 1$ skeletons are sufficient to provide a $(r,\eps)$-ID in expectation~\citep{belabbas2009spectral,guruswami2012optimal,chen2022randomly}.
As a trade-off, the classical SVD-based algorithm for k-DPP$(\Xb)$~\citep{hough2006determinantal} takes $O\rbr{nd^2}$ to construct the distribution and $O\rbr{nk^2}$ to draw $k$ samples.

In contrast to sampling from sophisticated distributions, uniform sampling works well for incoherent matrices (whose singular vectors distribute evenly along all canonical bases) but easily fails on coherent ones~\citep{cohen2015uniform}.
Although an in-depth comparison of different sampling methods is beyond the scope of this work, we refer the interested reader to the following enlightening references: \cite{deshpande2006matrix,cohen2015uniform,derezinski2021determinantal,chen2022randomly}.

\section{Combining Adaptiveness and Randomness}\label{sec:adaptive_and_random}
Recall that the vulnerability of greedy pivoting (\Cref{rmk:greedy_squared_norm_pivot}) comes from the scarce adversarial inputs. {Under randomization, the probability that greedy pivoting keeps making adversarial skeleton selections becomes negligible.} Meanwhile, the vulnerability of squared-norm sampling (\Cref{ex:adversary_squared_norm_sampling}) can be effectively circumvented by incorporating adaptive updates.  
Therefore, a natural question is \emph{how to combine adaptiveness and randomness effectively for better skeleton selection}.
In this section, we review two existing ideas that involve both adaptiveness and randomness in different ways, while comparing their respective advantages and drawbacks.

\subsection{Random Pivoting}
Random pivoting~\citep{deshpande2006matrix,deshpande2006adaptive,arthur2007k,chen2022randomly} is arguably the most intuitive skeleton selection method that combines adaptiveness and randomness. It can be simply described as ``replacing the greedy squared-norm pivoting in CPQR with squared-norm sampling'' as formalized in \Cref{algo:pqr_id}, with an asymptotic complexity of {$O(ndk)$}, consisting of $k$ inherently sequential passes through $\Xb$ (or its submatrices). The storage of $\Lb^{(t)}$ and $\Qb^{(t)}$ requires $O(nk)$ and $O(dk)$ memory. Analogous to greedy pivoting (\Cref{rmk:cpqr_update}), \Cref{algo:pqr_id} is {rank-adaptive} thanks to its adaptive nature.

\begin{algorithm}[!t]
    \caption{Sequential random pivoting (SRP)~{\citep[Algorithm 7]{chen2022randomly}}}\label{algo:pqr_id}
    \begin{algorithmic}[1]
        \Require (a) Data matrix $\Xb = \sbr{\xb_1,\cdots,\xb_n}^\top \in \R^{n \times d}$, 
        (b) Relative error tolerance $\tau = \rbr{1+\eps}\eta_r \in (0,1)$ (or target rank $k \in \N$), 

        \Ensure (a) Skeleton indices $S = \cbr{s_1,\cdots,s_k} \subset [n]$, 
        (b) $\Lb \in \R^{n \times k}$,
        (c) $\pib \in \Sfrak_n$
    
        \State $\Lb^{(0)} \gets \b{0}_{n \times 0}$, $\Qb^{(0)} \gets \b{0}_{d \times 0}$, $\pib \gets [1,\cdots,n]$, $S \gets \emptyset$, $t = 0$
        \State $\db^{(0)} \in \R_{\ge 0}^n$ with $\db^{(0)}(i) = \nbr{\xb_i}_2^2 ~\forall~ i \in [n]$ \Comment{{$O(nd)$}}
        \While{$\nbr{\db^{(t)}}_1 < \tau \nbr{\db^{(0)}}_1$ (or $\abbr{S} < k$)}
            \State $t \gets t + 1$
            \State $s_t \sim \db^{(t-1)} / \nbr{\db^{(t-1)}}_1$ 
    
            \State $S \gets S \cup \cbr{s_t}$, $\pib \gets \mathtt{swap}\rbr{\pib, t, s_t}$ \Comment{$\pib(1:t)$ consists of $S$}
            \State $\ab^{(t)} \gets \b{0}_{n}$, $\pib_a \gets \pib(t:n)$ \Comment{``a'' for active}
    
            \State $\vb \gets \xb_{s_t} - \Qb^{(t-1)} \rbr{\rbr{\Qb^{(t-1)}}^\top \xb_{s_t}} \in \R^d$ \Comment{{$O(dt)$}}
            \State $\Qb^{(t)} \gets \sbr{\Qb^{(t-1)}, \vb / \nbr{\vb}_2} \in \R^{d \times t}$ \Comment{{$O(d)$}}
            \State $\ab^{(t)}\rbr{\pib_a} \gets \Xb\rbr{\pib_a,:} \vb$ \Comment{{$O\rbr{nd}$}}
            
            \State $\Lb^{(t)} \gets \sbr{\Lb^{(t-1)}, \ab^{(t)} / \sqrt{\ab^{(t)}\rbr{s_t}}}$ \Comment{$\ab^{(t)}\rbr{s_t} = \db^{(t-1)}(s_t)$}
            \State $\db^{(t)}(i) \gets 0 ~\forall~ i \in S$, $\db^{(t)}(i) \gets \db^{(t-1)}(i) - \frac{\rbr{\ab^{(t)}(i)}^2}{\ab^{(t)}\rbr{s_t}} ~\forall~ i \notin S$ \Comment{{$O(n)$}}
        \EndWhile
        \State $k \gets \abbr{S}$, $\Lb \gets \Lb^{(t)}$
    \end{algorithmic} 
\end{algorithm}

The idea of random pivoting is ubiquitous in various related problems, \eg the combination of volume sampling and adaptive sampling for low-rank approximation~\citep{deshpande2006adaptive}, the Randomly Pivoted Cholesky (RPCholesky)~\citep{chen2022randomly} for kernel column Nystr\"om approximation~\citep[19.2]{martinsson2020randomized}, and the $D^2$-sampling for \texttt{k-means++} clustering~\citep{arthur2007k}. 
Specifically, {in light of the close connection between partially pivoted QR and pivoted Cholesky factorization (see \eg \citep[\S 5.2]{higham1990analysis} and \cite{musco2017sublinear,chen2022randomly})}, applying RPCholesky to $\Xb \Xb^\top \in \R^{n \times n}$ is equivalent to \Cref{algo:pqr_id} in exact arithmetic.
In light of such equivalence, {\cite[Theorem 5.1]{chen2022randomly}} provides a nearly optimal skeleton complexity guarantee for sequential random pivoting:
\begin{proposition}[Sequential random pivoting~\citep{chen2022randomly}]\label{prop:sas_skeleton_complexity}
    For any $\eps > 0$ and $r \in [n]$, the skeleton subset $S$ selected by \Cref{algo:pqr_id} provides $\E\sbr{\errskel{\Xb}{S}} \le \rbr{1 + \eps} \nbr{\Xb - \tsvd{\Xb}{r}}_F^2$ when\footnote{
        Without strictly specifying, we generally assume that $\eps > 0$ is small enough such that the $(r,\eps)$-ID is a non-trivial approximation, \ie $\eps < \min\cbr{1/\eta_r, 2^r}$ such that both logarithmic terms are positive. Otherwise, the skeleton complexity guarantee still holds by replacing any negative logarithmic terms with zeros.
    }
    \begin{align*}
        k = \abbr{S} \ge \frac{r}{\eps} + r \cdot \min\cbr{\log\rbr{\frac{1}{\eps \eta_r}}, 1 + \log\rbr{\frac{2^r}{\eps}}}.
    \end{align*}
\end{proposition}

Although random pivoting combines the strength of both adaptiveness and randomness and achieves an attractive skeleton complexity, the lack of {hardware efficiency} due to its sequential nature becomes a critical concern when considering empirical runtime.
\begin{remark}[Inefficiency of sequential updates]
    %
    {Given a sequence of matrix operations, it is well-known that an appropriate implementation using Level-3 BLAS, or matrix-matrix, operations will run more efficiently on modern processors than an optimal implementation using Level-2 or Level-1 BLAS~\cite{blackford2002updated}. This is largely due to the greater potential for the Level-3 BLAS to make more efficient use of memory caching in the processor.
    Notice that the computational bottleneck of \Cref{algo:pqr_id} with asymptotic complexity $O\rbr{ndk}$ consists of $k$ matrix-vector multiplications (Level-2 BLAS) with $\Xb$. This leads to the considerable slowdown of \Cref{algo:pqr_id} in practice.}     
\end{remark}

\subsection{Sketchy Pivoting}
Alternative to random pivoting, randomness can be combined with adaptiveness through randomized linear embedding (also known as sketching), which leads to the sketchy pivoting methods~\citep{martinsson2017householder,voronin2017efficient,dong2021simpler,2017_blockQR_ming_article}.
The general framework of sketchy pivoting~\citep[Algorithm 1]{dong2021simpler} consists of two stages. 
In the first stage, \emph{randomness is incorporated through sketching}, which serves two purposes. 
\begin{itemize}
    \item From the computational efficiency perspective, sketching reduces the data dimension\footnote{
        For applications like low-rank approximations and skeleton selection, constant oversampling like $l \ge k+10$ is usually sufficient in practice. 
        We refer the interested readers to \cite{halko2011finding,woodruff2014sketching,martinsson2020randomized} for a general review of sketching.
    } 
    from $d$ to $l = O(k)$ and transfers the $O(ndk)$ computational bottleneck from sequential greedy pivoting to input-sparsity-time matrix-matrix multiplications. 
    \item From the skeleton complexity perspective, applying randomized embeddings to the row space of $\Xb$ via sketching improves the empirical robustness of the subsequent greedy pivoting~\citep{trefethen1990average,dong2021simpler}. Intuitively, this is because the scarce adversarial inputs for greedy pivoting (\Cref{rmk:greedy_squared_norm_pivot}) are effectively negligible under randomization.
\end{itemize}

    

In the second stage, \emph{adaptiveness is utilized as greedy pivoting on the sketched sample matrix} $\Yb = \Xb\Omegab \in \R^{n \times l}$ to select skeletons. Both LUPP~\citep[Section 3.4]{golub2013matrix} and CPQR~\citep[Section 5.4.2]{golub2013matrix} on $\Yb$ cost $O(nk^2)$ asymptotically, whereas LUPP is remarkably faster in practice due to its superior parallelizability~\citep{dong2021simpler}. Intuitively, applying greedy pivoting on the top of sketching can be viewed as an analog of the squared-norm sampling with a slightly different squared-norm-based distribution.
\begin{remark}[Sketchy pivoting v.s. random pivoting]\label{rmk:sketchy_pivoting_rand_pivoting}
    For example, consider a Gaussian random matrix $\Omegab$ with $\iid$ entries from $\Ncal\rbr{0, 1/l}$. Each row $\yb_i = \Omegab^\top \xb_i$ of $\Yb$ is a Gaussian random vector with $\iid$ entries from $\Ncal(\b0, \nbr{\xb_i}_2^2 \Ib_l)$.
    When applying greedy squared-norm pivoting (CPQR) in the second stage, the first pivot is selected according to $s_1 \gets \argmax_{i \in [n]} \nbr{\xb_i}_2^2 Z_i$ where every $Z_i = \sum_{j=1}^l Y_{ij}^2$ is a $\chi^2_l$ random variable but $\csepp{Z_i}{i \in [n]}$ are dependent. The intuition behind sketchy pivoting is similar to random pivoting, but the randomness arises from the random variables $Z_i$, instead of sampling.

    In the extreme scenarios, when the variance of $\cbr{Z_i}_{i \in [n]}$ is negligible in comparison to the squared norms $\{\nbr{\xb_i}_2^2\}_{i \in [n]}$, sketchy pivoting behaves similarly to greedy pivoting. Meanwhile, when the variance of $\cbr{Z_i}_{i \in [n]}$ is much larger than differences among the squared norms, sketchy pivoting tends to behave like uniform sampling. Intuitively, there exists some $\cbr{Z_i}_{i \in [n]}$ with appropriate variance interpolating the two extreme cases such that sketchy pivoting mimics the behavior of squared-norm sampling and enjoys a nearly optimal skeleton complexity guarantee similar to \Cref{prop:sas_skeleton_complexity}.
\end{remark}

In practice, sketchy pivoting provides high-quality skeleton selection with comparable error $\errskel{\Xb}{S}$ to random pivoting (\Cref{algo:pqr_id}, \cf \Cref{fig:id_gmm_rpgp}), but much more efficiently thanks to the parallelizability of sketching~\citep{dong2021simpler}. 
However, as a major empirical limitation, sketchy pivoting is not {rank-adaptive} and requires prior knowledge of the target rank $k$.

\section{Robust Blockwise Random Pivoting}\label{sec:rbas}
In light of the existing skeleton selection methods summarized in \Cref{tab:adaptiveness_randomness_summary}, a \emph{{hardware-efficient}, {rank-adaptive}, and ID-revealing} skeleton selection algorithm that attains similar skeleton and asymptotic complexities as the sequential random pivoting raises as a desirable missing piece.

Blockwise random pivoting is a natural extension of its sequential variation \Cref{algo:pqr_id} that consists of matrix-matrix multiplications and is, therefore, {hardware-efficient}. In the kernel formulation (with an input kernel matrix $\Xb\Xb^\top$), {\cite[Algorithm 3]{chen2022randomly}} introduced a blocked version of RPCholesky that can be generalized as blockwise random pivoting (BRP---\Cref{algo:bg_pqr_id} with $\tau_b = 0$). 
{However, when the block size $b \in \mathbf{N}$ is taken to be large enough to fully realize the advantages of hardware efficiency, such plain BRP can have up to a $b$ times larger skeleton complexity than necessary due to the similar pitfall of squared-norm sampling instantiated in \Cref{ex:adversary_squared_norm_sampling} (\eg \Cref{fig:id_gmm_rpgp}).}

\begin{algorithm}[!ht]
\caption{Robust blockwise random pivoting (RBRP)}\label{algo:bg_pqr_id}
\begin{algorithmic}[1]
    \Require (a) $\Xb = \sbr{\xb_1,\cdots,\xb_n}^\top \in \R^{n \times d}$, 
    (b) $\tau = \rbr{1+\eps} \eta_r \in (0,1)$ (or $k \in \N$), 
    (c) Block size $b \in \N$,
    (d) $\tau_b \in [0,1)$---Tolerance for blockwise filtering (we choose $\tau_b = \frac{1}{b}$)
    \Ensure (a) $S = \cbr{s_1,\cdots,s_k} \subset [n]$, 
    (b) $\Lb \in \R^{n \times k}$,
    (c) $\pib \in \Sfrak_n$
    
    \State $\Lb^{(0)} \gets \b{0}_{n \times 0}$, $\Qb^{(0)} \gets \b{0}_{d \times 0}$, $\pib \gets [1,\cdots,n]$, $S \gets \emptyset$, $t = 0$
    \State $\db^{(0)} \in \R_{\ge 0}^n$ with $\db^{(0)}(i) = \nbr{\xb_i}_2^2 ~\forall~ i \in [n]$ \Comment{{$O(nd)$}}
    \While{$\nbr{\db^{(t)}}_1 > \tau \nbr{\db^{(0)}}_1$ (or $\abbr{S} < k$)}
        \State $t \gets t + 1$, $b \gets \min(b, k-\abbr{S})$
        \State Sample $S_t = \rbr{s_{\abbr{S}+1}, \cdots, s_{\abbr{S}+b}} \sim \db^{(t-1)} / \nbr{\db^{(t-1)}}_1$ without replacement \Comment{{$O(n)$}} 

        \State $\Vb \gets \Xb\rbr{S_t,:}^\top - \Qb^{(t-1)} \rbr{\rbr{\Qb^{(t-1)}}^\top \Xb\rbr{S_t,:}^\top} \in \R^{d \times b}$ \Comment{{$O(db\abbr{S})$}}
        \State $\Qb_\Vb, \Rb_\Vb, \pib_\Vb \gets \mathtt{qr}\rbr{\Vb,~ \t{``econ''},~ \t{``vector''}}$ \Comment{$\Qb_\Vb \in \R^{d \times b}$, $\Rb_\Vb \in \R^{b \times b}$, {$O\rbr{db^2}$}} 
        \State $b' \gets \max\csepp{i \in [b]}{\nbr{\Rb_\Vb\rbr{i:b, i:b}}_F^2 \ge \tau_b \cdot \nbr{\Rb_\Vb}_F^2}$ \Comment{{$O\rbr{b^2}$}}
        \State $S'_t \gets S_t\rbr{\pib_\Vb\rbr{1:b'}}$, $\Qb'_\Vb \gets \Qb_\Vb\rbr{:,1:b'}$
        \State $\Qb^{(t)} \gets \sbr{\Qb^{(t-1)}, \Qb'_\Vb} \in \R^{d \times \rbr{\abbr{S} + b'}}$ 

        \State $\pib \gets \mathtt{swap}\rbr{\pib, \abbr{S}+1:\abbr{S}+b', S'_t}$, $\pib_a \gets \pib(\abbr{S}+1:n)$
        \State $\wh\Lb^{(t)} \gets \b{0}_{n \times b'}$
        \State $\wh\Lb^{(t)}\rbr{\pib_a,:} \gets \Xb\rbr{\pib_a,:} \Qb'_\Vb$ \Comment{{$O\rbr{ndb'}$}}

        \State $S \gets S \cup S'_t$ 
        \State $\Lb^{(t)} \gets \sbr{\Lb^{(t-1)}, \wh\Lb^{(t)}} \in \R^{n \times \abbr{S}}$ 

        \State $\db^{(t)}(i) \gets 0 ~\forall~ i \in S$, $\db^{(t)}(i) \gets \db^{(t-1)}(i) - \nbr{\wh\Lb^{(t)}\rbr{i,:}}_2^2 ~\forall~ i \notin S$ \Comment{{$O(nb')$}}
    \EndWhile
    \State $k \gets \abbr{S}$, $\Lb \gets \Lb^{(t)}$
\end{algorithmic} 
\end{algorithm} 

To exemplify the pitfall of plain BRP, we present an adversarial input in \Cref{ex:gmm_pitfall_plain_bas} based on a tailored Gaussian mixture model (GMM).
\begin{example}[Pitfall of plain blockwise random pivoting]\label{ex:gmm_pitfall_plain_bas}
    For $n, d, k \in \N$ such that $n/k = m \in \N$, we draw $n$ points from a GMM with means $\Ccal = \cbr{\mub_j}_{j \in [k]}$, covariance $\Ib_d$, and cluster size $m$---$\Xcal = \cbr{\xb_i \in \R^d}_{i \in [n]} = \bigcup_{j=1}^{k} \Xcal_j$ where
    \begin{align*}\Xcal_j = \csepp{\xb_{m(j-1)+\iota} = \mub_j + \xib_\iota}{\mub_j = 10 j \cdot \eb_j,~ \xib_\iota \sim \Ncal\rbr{\b0_d, \Ib_d}~\iid~\forall~\iota \in [m]}.
    \end{align*}
    Consider a GMM data matrix $\Xb = \sbr{\xb_1,\cdots,\xb_n}^\top \in \R^{n \times d}$. Since the means in $\Ccal$ have distinct norms whose discrepancies dominate the covariance $\Ib_d$, $\Xcal$ can be partitioned into $k$ clusters $\cbr{\Xcal_j}_{j \in [k]}$ each containing $m$ points with distinct norms. 
    
    The best size-$k$ skeleton subset consists of exactly one point from each cluster $\Xcal_j$, while multiple points from the same cluster are redundant.
    However, the plain blockwise adaptiveness (random (BRP)/greedy pivoting (BGP)) tends to pick multiple points from the same cluster. For instance, with block size $b$, BGP can pick up to $\min\rbr{b,m}$ points in the same cluster. 
    Similar inefficiency also appears in BRP but is generally alleviated thanks to randomness (\Cref{fig:id_gmm_rpgp} (left)).
\end{example}

To resolve the challenges illustrated by \Cref{ex:gmm_pitfall_plain_bas}, we introduce a \emph{robust blockwise random pivoting (RBRP)} algorithm---\Cref{algo:bg_pqr_id}---that empirically achieves comparable skeleton complexities to the sequential adaptive methods (SRP and CPQR), as demonstrated in \Cref{fig:id_gmm_rpgp} (right), while exploiting the {hardware efficiency} of matrix-matrix multiplications.

\begin{remark}[Robust blockwise filtering]\label{rmk:robust_blockwise_filtering}
    The key step in \Cref{algo:bg_pqr_id} that improves the robustness of BRP is the robust blockwise filtering with tolerance $\tau_b \in [0,1]$---applying truncated CPQR locally to the small residuals of the selected candidates $\Vb \in \R^{d \times b}$:
    \begin{align*}
        \Vb\rbr{:,\pib_{\Vb}} = \Qb_{\Vb} \Rb_{\Vb} \approx \Qb_{\Vb}\rbr{:,1:b'} \Rb_{\Vb}\rbr{1:b',:}
    \end{align*}
    where the relative truncation error is upper-bounded by $\tau_b$
    \begin{align*}
        \nbr{\Vb\rbr{:,\pib_{\Vb}}-\Qb_{\Vb}\rbr{:,1:b'} \Rb_{\Vb}\rbr{1:b',:}}_F^2 = \nbr{\Rb_{\Vb}\rbr{b'+1:b,b'+1:b}}_F^2 < \tau_b \nbr{\Vb}_F^2.
    \end{align*}
    With small constant block sizes, the robust blockwise filtering based on local CPQR can be computed with negligible additional cost $O(db^2)$, despite the sequential nature of CPQR.
\end{remark}

{The worst-case asymptotic complexity of \Cref{algo:bg_pqr_id} is $O\rbr{ndk + bdk^2}$ (with a slightly larger lower-order term, \cf $O\rbr{ndk + dk^2}$ for SRP in \Cref{algo:pqr_id}) when each step of robust blockwise filtering in \Cref{algo:bg_pqr_id} only keeps a single skeleton.
In practice, robust blockwise filtering typically retains a reasonable portion of skeletons in each block, leading to a time complexity close to $O\rbr{ndk + dk^2}$ (see \Cref{subsec:id_accuracy_efficiency}).}
{The dominant cost $O(ndk)$ consists of $k$ hardware-efficient passes through $\Xb$ as matrix-matrix multiplications. Analogous to \Cref{algo:pqr_id}, the storage of $\Lb^{(t)}$ and $\Qb^{(t)}$ requires $O(nk)$ and $O(dk)$ memory. The adaptive nature of \Cref{algo:bg_pqr_id} further facilitates the rank-adaptiveness.}

\begin{table}[!h]
    \centering
    \vspace{-1em}
    \caption{Summary of abbreviations for the skeleton selection methods in the experiments.}\label{tab:abbrev_skeleton_selection}
    \begin{tabular}{c|cc}
    \toprule
    Abbreviation & Skeleton selection method & Algorithm \\
    \midrule
    CPQR & Sequential greedy pivoting & \cite[Section 5.4.2]{golub2013matrix} \\
    SqNorm & Squared-norm sampling & \eqref{eq:squared_norm_distribution}~\citep{deshpande2006matrix} \\
    DPP & DPP sampling & \cite[Equation 8]{belabbas2009spectral}\cite{kulesza2011k,GPBV19} \\
    SRP & Sequential random pivoting & \Cref{algo:pqr_id}~\citep{chen2022randomly} \\
    SkCPQR & Sketchy pivoting with CPQR & \cite[Algorithm 4]{voronin2017efficient} \\
    SkLUPP & Sketchy pivoting with LUPP & \cite[Algorithm 1]{dong2021simpler} \\
    BGP & Blockwise greedy pivoting & \Cref{rmk:blockwise_gp} with $\tau_b = 0$ \\
    BRP & Blockwise random pivoting & \Cref{algo:bg_pqr_id} with $\tau_b = 0$ \\
    \midrule
    RBGP & Robust blockwise greedy pivoting & \Cref{rmk:blockwise_gp} with $\tau_b = \frac{1}{b}$ \\
    \b{RBRP} & Robust blockwise random pivoting & \Cref{algo:bg_pqr_id} with $\tau_b = \frac{1}{b}$ \\
    \bottomrule
    \end{tabular}
\end{table}

\begin{remark}[Blockwise greedy pivoting]\label{rmk:blockwise_gp}
    For the completeness of comparison, in the experiments, we consider a variation of the (robust) blockwise random pivoting ((R)BRP)---(robust) blockwise greedy pivoting ((R)BGP). In \Cref{algo:bg_pqr_id}, (R)BGP replaces the sampling step,
    \begin{center}
        sample $S_t = \rbr{s_{\abbr{S}+1}, \cdots, s_{\abbr{S}+b}} \sim \db^{(t-1)} / \nbr{\db^{(t-1)}}_1$ without replacement,
    \end{center}
    with the task of choosing points corresponding to the top-$b$ probabilities greedily:
    \begin{center}
        select $S_t = \rbr{s_{\abbr{S}+1}, \cdots, s_{\abbr{S}+b}} \gets$ indices of the top-$b$ entries in $\db^{(t-1)}$.
    \end{center}
    With $\tau_b = 0$, BGP is reduced to CPQR. 
    For adversarial inputs like the one in \Cref{ex:gmm_pitfall_plain_bas}, BGP tends to suffer from far worse skeleton complexities than those of BRP (\Cref{fig:id_gmm_rpgp} (left)). 
\end{remark}

\begin{figure}[!ht]
    \includegraphics[width=\textwidth]{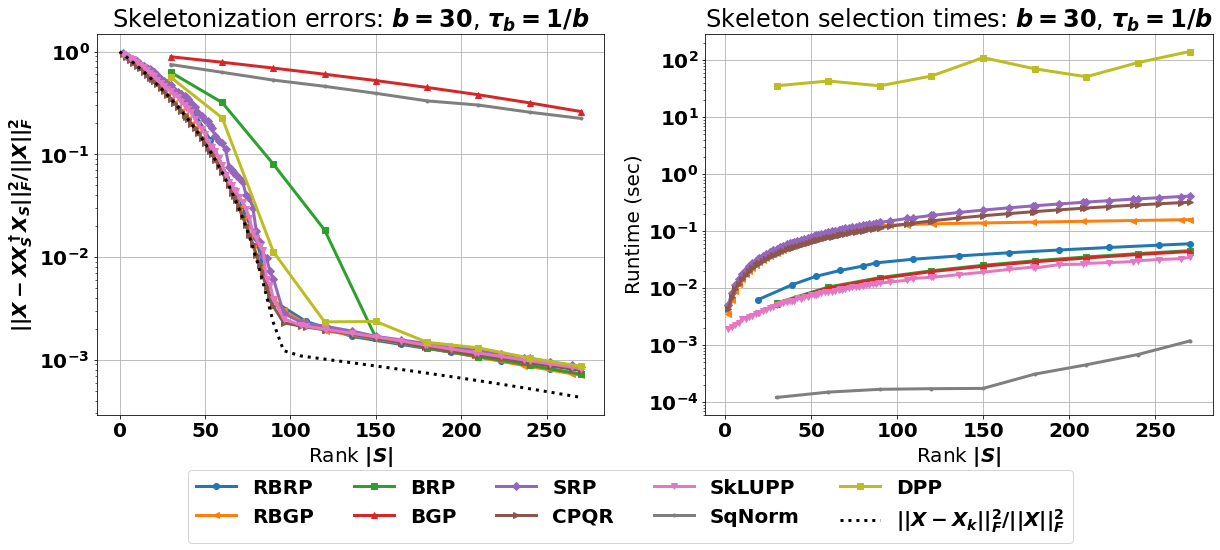}
    \caption{
        {The relative skeletonization error $\errskel{\Xb}{S}/\|\Xb\|_F^2$ and skeleton selection time (excluding interpolation matrix construction) of algorithms in \Cref{tab:abbrev_skeleton_selection} on the adversarial \texttt{GMM} matrix described in \Cref{subsec:data_matrices} constructed according to \Cref{ex:gmm_pitfall_plain_bas}. For the sketchy pivoting methods, we only show the results of SkLUPP, as SkCPQR is known to have similar skeleton complexities as SkLUPP but with higher runtimes~\citep{dong2021simpler}.}
    }\label{fig:id_gmm_rpgp}
    \vspace{-1em}
\end{figure}

{\Cref{fig:id_gmm_rpgp} concretizes the adversarial input in \Cref{ex:gmm_pitfall_plain_bas} and demonstrates the empirical effectiveness of RBRP for skeleton selection.
\begin{itemize}
    \item In particular, the skeletonization errors on the left unveil the failure of plain \b{BRP/BGP} in terms of large skeleton complexities due to the lack of adaptiveness. Moreover, greedy selection in BGP tends to severely exacerbate this problem. \b{SqNorm} can be viewed as plain BRP with block size $b=k$ and therefore suffers from similar large skeleton complexities as BRP/BGP, despite its low asymptotic complexity of $O(nd)$.
    \item \b{DPP} sampling attains relatively good accuracy even without adaptiveness, thanks to its nearly optimal skeleton complexity~\citep{belabbas2009spectral,guruswami2012optimal,chen2022randomly}, but with much higher runtimes as a trade-off (\eg, with \cite{GPBV19}, k-DPP takes about $10^3 \times$ longer than other algorithms in \Cref{tab:abbrev_skeleton_selection}).
    \item Providing even better skeleton complexities than DPP, \b{SRP} and \b{CPQR} achieve one of the best skeletonization errors thanks to their adaptiveness. However, the sequential nature of SRP and CPQR leads to higher runtimes than the blockwise methods.
    \item With robust blockwise filtering (\Cref{rmk:robust_blockwise_filtering}), \b{RBRP} achieves comparable skeleton complexities to SRP and CPQR, while maintaining the hardware efficiency of blockwise methods. Notice from the runtimes on the right that despite the comparable skeleton complexities to RBRP, \b{RBGP} is slower than RBRP on the adversarial \texttt{GMM} matrix. This is because RBGP with greedy selection tends to pick more redundant points in each block, which are later removed by the robust blockwise filtering.
    \item For skeleton selection only, \b{SkLUPP} provides the state-of-the-art performance --- one of the best skeletonization errors with the most competitive runtimes. However, the lack of rank-adaptiveness and ID-revealing property becomes a major drawback for SkLUPP as an ID algorithm (see \Cref{subsec:inexact_id_revealing_algo}).
\end{itemize}}

Beyond the \texttt{GMM} adversarial input, extensive numerical experiments in \Cref{sec:experiments} show that RBRP with $\tau_b = \frac{1}{b}$ empirically attains comparable skeleton complexities as sequential random pivoting, which leads to the following conjecture.
\begin{conjecture}[Skeleton complexity of RBRP]\label{conj:rbas_skeleton_complexity}
    RBRP (\Cref{algo:bg_pqr_id}) with $\tau_b = \frac{1}{b}$ has similar skeleton complexity as sequential random pivoting (\Cref{prop:sas_skeleton_complexity}) in expectation.
\end{conjecture}

\begin{proof}[Rationale for \Cref{conj:rbas_skeleton_complexity}]
    Consider \Cref{algo:bg_pqr_id} at the $t$-th step with $\abbr{S}$ selected skeletons from the previous steps, $\abbr{S_t} = b$ candidates drawn from ${\db^{(t-1)}}/{\nbr{\db^{(t-1)}}_1}$, and $\abbr{S'_t} = b'$ remaining skeletons in $S_t$ that pass through the robust blockwise filtering (\Cref{rmk:robust_blockwise_filtering}).
    Now suppose we switch to sequential random pivoting after the $(t-1)$-th step and draw $b'$ skeletons adaptively, denoted as $S''_t$. 
    Then, the key rationale for \Cref{conj:rbas_skeleton_complexity} is that $\errskel{\Xb}{S \cup S'_t} \approx \errskel{\Xb}{S \cup S''_t}$ with (a reasonably) high probability when $\tau_b = \frac{1}{b}$.

    On this, we observe the close connection between the selection scheme of $S'_t$ and that of $S''_t$: 
    \begin{itemize}
        \item The robust blockwise filtering for $S'_t$ in RBRP adaptively selects the point with the maximum residual norm in each step (via CPQR on $\Vb$ in \Cref{algo:bg_pqr_id}), within a small subset $S_t$ of $b$ points sampled according to the squared-norm distribution over $[n] \setminus S$.
        \item The random pivoting for $S''_t$ in SRP adaptively selects the point according to the squared-norm distribution of the residual over the complement of the skeleton subset in each step.
    \end{itemize}
    Intuitively, every selection made for $S'_t$ and $S''_t$ yields similar decay in the skeletonization error if the squared norms of the residual matrix, corresponding to remaining points in $S_t$, are around the maxima in the complement of the skeleton subset so that these remaining points in $S_t$ can be selected by SRP with high probability\footnote{
        It is worth mentioning that, in contrast to SRP, RBRP does not recompute the entire residual matrix after each skeleton selection. Instead, within each block, RBRP computes the residual matrix locally for $\Xb\rbr{S_t,:}$ only, while updating the entire residual matrix at the end of each block via matrix-matrix multiplication.
    }. 
    Instead of computing the squared norms of the entire residual matrix explicitly for each selection as SRP (which leads to the undesired sequential updates), RBRP leverages the threshold $\tau_b = \frac{1}{b}$ to encourage the first $b'$ points selected by the robust blockwise filtering in each block to have top squared norms in the corresponding residual matrices. 
    Notice that with $\tau_b = \frac{1}{b}$, we enforce the sum of squared norms of the $b'$ points to be larger than the average squared norm of the original block: $\nbr{\Rb_\Vb\rbr{i:b, i:b}}_F^2 \ge \frac{1}{b} \nbr{\Vb}_F^2$.
    Since the squared norms of residuals are non-increasing under QR updates, $\frac{1}{b} \nbr{\Vb}_F^2$ serves as an overestimate for the squared norms of redundant points sampled in $S_t$.

    As toy examples, if $\Vb$ is rank deficient (\ie $\rank(\Vb) = b'' < b$), then there exists $b' \le b''$ such that $\nbr{\Rb_\Vb\rbr{b':b, b':b}}_F^2 \ge \frac{1}{b} \nbr{\Vb}_F^2$, but
    \begin{align*}
        0 = \nbr{\Rb_\Vb\rbr{b''+1:b, b''+1:b}}_F^2 \le \nbr{\Rb_\Vb\rbr{b'+1:b, b'+1:b}}_F^2 < \frac{1}{b} \nbr{\Vb}_F^2.
    \end{align*}
    Whereas if $\Vb \in \R^{d \times b}$ contains orthogonal columns with the same squared norm, then for all $i \in [b]$, we have $\nbr{\Rb_\Vb\rbr{i:b, i:b}}_F^2 \ge \abbr{\Rb_\Vb\rbr{b, b}}^2 = \frac{1}{b} \nbr{\Vb}_F^2$, and therefore $b'=b$.
\end{proof}

\section{Interpolation Matrix Construction}\label{sec:id_construction}

In this section, we consider the construction of the interpolation matrix $\Wb$ after skeleton selection. In particular, we are interested in leveraging information from the skeleton selection processes to construct $\Wb$ efficiently.

Since the optimal interpolation matrix solves the following least-square problem:
\begin{align}\label{eq:stable_id}
    \min_{\Wb \in \R^{n \times k}} \nbr{\Xb - \Wb \Xb_S}_F^2 
    \quad\Rightarrow\quad
    \Wb = \Xb \Xb_S^\pinv = \Xb \Xb_S^\top \rbr{\Xb_S \Xb_S^\top}^{-1},
\end{align}
we can compute such optimal $\Wb$ exactly in $O\rbr{ndk + dk^2 + nk^2}$ time via QR decomposition on $\Xb_S$. However, the dominant cost $O(ndk)$ requires additional passes through $\Xb$ and can be prohibitive after the skeleton selection stage.

\subsection{ID-revealing Algorithms}
With the output matrix $\Lb \in \R^{n \times k}$ from \Cref{algo:pqr_id,algo:bg_pqr_id}, the corresponding optimal interpolation matrix $\Xb \Xb_S^\dagger$ can be computed in $O(nk^2)$ time following \Cref{algo:id}.

\begin{algorithm}[!ht]
    \caption{Interpolation matrix construction (ID)}\label{algo:id}
    \begin{algorithmic}[1]
        \State $\Lb_1 \gets \Lb\rbr{\pib\rbr{1:k},:} \in \R^{k \times k}$, $\Lb_2 \gets \Lb\rbr{\pib\rbr{k+1:n},:} \in \R^{(n-k) \times k}$
        \State $\Wb \gets \b{0}_{n \times k}$, $\Wb\rbr{\pib\rbr{1:k},:} = \Ib_k$
        \State $\Wb\rbr{\pib\rbr{k+1:n},:} = \Lb_2 \Lb_1^{-1}$ \Comment{{$O\rbr{nk^2}$}}
\end{algorithmic} 
\end{algorithm} 

\begin{proposition}\label{prop:exact_id_revealing_algo}
    The sequential (\Cref{algo:pqr_id}) and blockwise (\Cref{algo:bg_pqr_id}) random pivoting, as well as their greedy pivoting variations, are ID-revealing (\Cref{def:id_revealing}).
\end{proposition}

\begin{proof}[Proof of \Cref{prop:exact_id_revealing_algo}]
    We first observe that the difference between sequential/blockwise random v.s. greedy pivoting lies only in the skeleton subset $S$. That is, assuming the same skeleton selection $S$, the resulting matrices $\Lb$ from random and greedy pivoting would be the same. Therefore, it suffices to show that \Cref{prop:exact_id_revealing_algo} holds for random pivoting in \Cref{algo:pqr_id,algo:bg_pqr_id}.

    For both \Cref{algo:pqr_id,algo:bg_pqr_id}, recall that $\Lb \gets \Lb^{(t)}$. We observe that $\Lb \Qb^{(t)\top}$ provides a rank-$k$ approximation for $\Xb$: $\nbr{\Xb - \Lb \Qb^{(t)\top}}_F^2 = \nbr{\db^{(t)}}_1 = \errskel{\Xb}{S} < \tau \nbr{\Xb}_F^2$. We will show that $\Lb \Qb^{(t)\top} = \Xb \Xb_S^\dagger \Xb_S$ in exact arithmetic.
    
    Then, with $\Lb_1 \gets \Lb\rbr{\pib\rbr{1:k},:}$ and $\Lb_2 \gets \Lb\rbr{\pib\rbr{k+1:n},:}$ as in \Cref{algo:id}, by constructing, $\Lb_2 \Qb^{(t)\top} = \Xb\rbr{\pib\rbr{k+1:n},:}$ and $\Lb_1 \Qb^{(t)\top} = \Xb_S = \Xb\rbr{\pib\rbr{1:k},:}$ where $\Lb_1 \in \R^{k \times k}$ is invertible as $\Xb_S$ consists of linearly independent rows. Therefore, $\Xb_S^\dagger = \Qb^{(t)} \Lb_1^{-1}$, and the optimal interpolation matrix $\Wb = \Xb \Xb_S^\dagger$ can be expressed up to permutation as
    \begin{align*}
        \Wb\rbr{\pib,:} = \Xb\rbr{\pib,:}\Xb_S^\dagger = \bmat{\Lb_1 \Qb^{(t)\top} \\ \Lb_2 \Qb^{(t)\top}} \Qb^{(t)} \Lb_1^{-1} = \bmat{\Lb_1 \\ \Lb_2} \Lb_1^{-1} = \bmat{\Ib_k \\ \Lb_2 \Lb_1^{-1}}.
    \end{align*}
    The computation of $\Lb_2 \Lb_1^{-1}$ involves solving a small $k \times k$ system $(n-k)$ times, which takes $O(nk^2)$ time in general.
    Overall, given $\Lb$ from \Cref{algo:pqr_id,algo:bg_pqr_id}, we have that \Cref{algo:id} constructs the optimal interpolation matrix $\Wb = \Xb \Xb_S^\dagger = \Lb \Lb_1^{-1}$.
\end{proof}

Moreover, when $\Lb_1$ is well-conditioned, the special lower triangular structure of $\Lb_1$ can be leveraged to further accelerate the construction of $\Wb$, \ie $\Lb_2 \Lb_1^{-1}$ can be evaluated via backward substitution in $O\rbr{(n-k)k^2}$ time. 
However in practice, $\Lb_1$ from the blockwise algorithm \Cref{algo:bg_pqr_id} with a reasonably large block size tends to be ill-conditioned, as elaborated in \Cref{rmk:block_cholid_triangularization}.

\begin{remark}[Stable evaluation of $\Lb_2\Lb_1^{-1}$ for (R)BRP]\label{rmk:block_cholid_triangularization}
    In \Cref{algo:bg_pqr_id}, the resulting $\Lb_1 = \Lb\rbr{\pib\rbr{1:k},:}$ tends to be ill-conditioned, especially with a large block size $b$. For numerical stability, instead of using backward substitution, we evaluate $\Lb_1^{-1}$ via truncated SVD with small singular value clipping (\ie truncating singular values below a given tolerance, \eg $10^{-12}$ in our experiments, and their associated singular vectors), which takes $O\rbr{nk^2 + k^3}$ time. We emphasize that compared to backward substitution, computing the pseudoinverse via truncated SVD of the small $k \times k$ matrix $\Lb_1$ has negligible additional cost in the common low-rank approximation setting with $k \ll n$.
\end{remark}

\subsection{Oversampled Sketchy ID for Sketchy Pivoting}\label{subsec:inexact_id_revealing_algo}
In contrast to random pivoting, sketchy pivoting~\citep[Algorithm 1]{dong2021simpler} fails to be ID-revealing, intuitively due to the loss of information in the sketching process.
Concretely, applying \Cref{algo:id} on sketchy pivoting leads to a sketched version of \eqref{eq:stable_id}:
\begin{align}\label{eq:sketched_least_square}
    \min_{\Wb \in \R^{n \times k}} \nbr{\Xb\wh\Omegab - \Wb \Xb_S\wh\Omegab}_F^2 
    \quad\Rightarrow\quad
    \Wb = \Xb \wh\Omegab \rbr{\Xb_S \wh\Omegab}^\dagger,
\end{align}
where $\wh\Omegab = \Omegab\rbr{:,1:k} \in \R^{d \times k}$ consists of the first $k$ columns in the randomized linear embedding $\Omegab \in \R^{d \times l}$ drawn in sketchy pivoting.
Unfortunately, such sketched least-square estimation~\citep{woodruff2014sketching,raskutti2016statistical,dobriban2019asymptotics} is known to be suboptimal as long as $l < d$~\citep{pilanci2016iterative}. 
Specifically for interpolation matrix construction, \eqref{eq:sketched_least_square} without oversampling leads to large interpolation error $\errid{\Xb}{S}{\Wb} \gg \errskel{\Xb}{S}$, as numerically illustrated in \Cref{subsec:id_accuracy_efficiency} (\cf SkLUPP-ID in \Cref{fig:id_gmm} for $\errid{\Xb}{S}{\Wb}$ v.s. SkLUPP in \Cref{fig:id_gmm_rpgp} for $\errskel{\Xb}{S}$ (right)).

\begin{algorithm}[!ht]
    \caption{Oversampled sketchy interpolation matrix construction (OSID)}\label{algo:fsid}
    \begin{algorithmic}[1]
        \State (Draw a randomized linear embedding $\Omegab \in \R^{d \times l}$ and compute $\Yb \gets \Xb \Omegab \in \R^{n \times l}$)
        \State $\Qb, \Rb \gets \mathtt{qr}\rbr{\Yb\rbr{\pib\rbr{1:k},:}^\top, \t{``econ''}}$ \Comment{$\Qb \in \R^{l \times k},~ \Rb \in \R^{k \times k}$, {$O\rbr{lk^2}$}}
        \State $\Wb \gets \b{0}_{n \times k}$, $\Wb\rbr{\pib\rbr{1:k},:} = \Ib_k$
        \State $\Wb\rbr{\pib\rbr{k+1:n},:} \gets \rbr{\Yb\rbr{\pib\rbr{k+1:n},:} \Qb} \Rb^{-
        \top}$ \Comment{{$O\rbr{(n-k) \rbr{lk + k^2}}$}}
\end{algorithmic} 
\end{algorithm} 

As an effective remedy, the oversampled sketchy ID (OSID) in \Cref{algo:fsid} generalizes \Cref{algo:id} by increasing the sample size beyond $k$ (\ie oversampling), which is shown to considerably alleviate such suboptimality of \eqref{eq:sketched_least_square} in interpolation error $\errid{\Xb}{S}{\Wb}$ (\cf \Cref{subsec:id_accuracy_efficiency}, SkLUPP-ID v.s. SkLUPP-OSID in \Cref{fig:id_gmm,fig:id_mnist,fig:id_cifar10,fig:id_Gaussian_exp,fig:snn_a10,fig:helmholtz}). Concretely, with a moderate multiplicative oversampling $l = \Theta(k) > k$ (\eg $l = 3k$ is usually sufficient in practice, see \Cref{subsec:id_accuracy_efficiency}), \Cref{algo:fsid} can be expressed as
\begin{align}\label{eq:fsid}
\t{OSID}: \qquad
\begin{split}
    &\underset{k \times l}{\Yb_S} = \Yb\rbr{S,:}, \qquad \underset{l \times k}{\Yb_S^\top} = \underset{l \times k}{\Qb}\ \underset{k \times k}{\Rb} \\
    &\Wb = \Yb \Yb_S^\dagger = \rbr{\Yb \Qb} \Rb^{-\top} = \Xb \rbr{\Omegab \Omegab^\dagger} \Xb_S^\pinv 
\end{split}
\end{align}
where $\E_{\Omegab}\sbr{\Wb} = \Xb \Xb_S^\dagger$ is an unbiased estimate for the optimal interpolation matrix when the embedding $\Omegab$ is isotropic, \ie, $\E_{\Omegab}\sbr{\Omegab \Omegab^\dagger} = \Ib_d$. 

\Cref{algo:fsid} takes $O(nlk + nk^2 + lk^2) = O(nk^2)$ time. Compared to \Cref{algo:id} without oversampling, the dominant cost $O(nlk)$ increases proportionally to the oversampling factor $l/k$. Nevertheless, such additional cost is affordable in practice (\cf \Cref{fig:id_gmm,fig:id_mnist,fig:id_cifar10,fig:id_Gaussian_exp,fig:snn_a10,fig:helmholtz}) thanks to the {hardware efficiency} of matrix-matrix multiplications~\citep{goto2008high}.

{
While sketchy pivoting without oversampling in \eqref{eq:sketched_least_square} is not ID-revealing, the existing theories for sketch-and-solve (see \eg \cite[\S 2.5]{woodruff2014sketching}) implies that OSID in \eqref{eq:fsid} with Gaussian embedding and a multiplicative oversampling $l = \Theta(k)$ is $\gamma$-ID-revealing (recall \Cref{def:id_revealing}) for $\gamma = (1+O(\sqrt{k/l}))^2$ with a constant probability:
\begin{proposition}[Sketchy pivoting + OSID is $\gamma$-ID-revealing]\label{prop:fsid}
    Selecting a skeleton subset $S$ of size $k$ via sketchy pivoting~\citep[Algorithm 1]{dong2021simpler} outputs
    (i) a sketched sample matrix $\Yb = \Xb \Omegab \in \R^{n \times l}$ constructed by a Gaussian embedding $\Omegab \in \R^{d \times l}$ with $l > k$ and $\iid$ entries from $\Ncal(0,1/l)$, along with 
    (ii) a permutation $\pib \in \Sfrak_n$ that characterizes the skeleton subset $S = \pib\rbr{1:k}$.
    For any $\delta \in (0,1)$, when $l \gtrsim k + \log(1/\delta)$, with probability at least $1-\delta$, \Cref{algo:fsid} constructs an interpolation matrix $\Wb$ that satisfies $\errid{\Xb}{S}{\Wb} \leq (1 + O(\sqrt{\delta^{-1} k/l}))^2 \errskel{\Xb}{S}$ in $O(nkl)$ time.
\end{proposition}
\begin{proof}[Proof of \Cref{prop:fsid}]
    The proof of \Cref{prop:fsid} is a direct generalization of \cite[Theorem 23]{woodruff2014sketching}. 
    The key fact is that for a $(1/2, \delta/2, k)$-JLT~\citep[Definition 3]{woodruff2014sketching} $\Omegab \in \R^{d \times l}$ with $(\Theta(\eps/\sqrt{k}), \delta/2)$-JL second moment property~\citep[Definition 12]{woodruff2014sketching} for some $\eps > 0$, \eqref{eq:fsid} provides an interpolation matrix $\Wb$ that satisfies $\|\Wb\Xb_S - \Xb\Xb_S^\dagger\Xb_S\|_F \le \eps \|\Xb - \Xb\Xb_S^\dagger\Xb_S\|_F$ (see \eg the proofs of \cite[Theorem 23]{woodruff2014sketching} or \citep[Lemma C.5]{dong2024sketchy}), and therefore, $\errid{\Xb}{S}{\Wb} \leq (1 + \eps)^2 \errskel{\Xb}{S}$.
    Then, it is sufficient to recall that a Gaussian embedding $\Omegab \in \R^{d \times l}$ with $l \gtrsim k + \log(1/\delta)$ is a $(1/2, \delta/2, k)$-Johnson-Lindenstrauss transform~\citep[Theorem 6]{woodruff2014sketching} and satisfies the $(\Theta(\eps/\sqrt{k}), \delta/2)$-JL second moment property when $l \gtrsim \delta^{-1} k / \eps^2$.
\end{proof}
Notice that \Cref{prop:fsid} can be easily extended to various structured randomized linear embeddings like the sparse embedding in \cite{cohen2016nearly} and the LESS embedding in \cite{derezinski2021sparse,chenakkod2024optimal} that offer better efficiency for constructing the sample matrix $\Yb$. 
}


\section{Experiments}\label{sec:experiments}
{The experiments in \Cref{fig:id_gmm_rpgp,fig:id_gmm,fig:id_mnist,fig:id_cifar10,fig:id_Gaussian_exp,fig:snn_a10,fig:helmholtz} are conducted on a 14-core Apple M3 processor, while experiments in \Cref{tab:gpu_runtime} are conducted on an Nvidia V100 GPU.}

\subsection{Data Matrices}\label{subsec:data_matrices}
We compare accuracy and efficiency of various ID algorithms on different types of data matrices, including {natural data matrices from image datasets and integral operators, as well as synthetic matrices with varied spectra}, outlined as follows. 
\begin{enumerate}
    \item We construct a synthetic data matrix $\Xb \in \R^{2000 \times 500}$ based on the adversarial GMM example described in \Cref{ex:gmm_pitfall_plain_bas} with $k=100$ cluster centers, denoted as \texttt{GMM}.
    \item Recall that the MNIST testing set~\citep{lecun1998mnist} consists of $10,000$ images of hand-written digits from $0$ to $9$. We denote \texttt{MNIST} as a data matrix consisting of $n = 1000$ random images sampled uniformly from the MNIST testing set where each row contains a flattened and \emph{normalized} $28 \times 28$ image such that $d = 784$. The nonzero entries take approximately $20\%$ of the matrix.
    \item The CIFAR-10 dataset~\citep{krizhevsky2009learning} consists of $60,000$ colored images of size $32 \times 32 \times 3$.
    We denote \texttt{CIFAR-10} as a data matrix consisting of $n = 1000$ random images sampled uniformly from the CIFAR-10 dataset where each row contains a flattened and \emph{normalized} image such that $d = 3072$. 
    \item Let \texttt{Gaussian-exp} be a random dense data matrix of size $1000 \times 1000$ with exponential spectral decay. We construct \texttt{Gaussian-exp} from its SVD $\Xb = \Ub \Sigmab \Vb^\top$ where $\Ub, \Vb \in \R^{1000 \times 1000}$ are random unitary matrices drawn from the Haar measure, and $\Sigmab = \diag\rbr{\sigma_1,\cdots,\sigma_{1000}}$ where $\sigma_i = 1$ for all $i \le 100$ and $\sigma_i = \max\cbr{0.8^{i-100}, 10^{-5}}$ for all $i > 100$. 
    \item \texttt{SNN} is a sparse non-negative (SNN) matrix~\citep{sorensen2016deim} of size $1000 \times 1000$ generated by
    \begin{equation}
        \label{eq:snn-def}
        \Xb = \sum_{i=1}^{100} \frac{10}{i} \ub_i \vb_i^T + \sum_{i=101}^{1000} \frac{1}{i} \ub_i \vb_i^T
    \end{equation}
    where $\Ub = \sbr{\ub_1,\cdots,\ub_{1000}}$ and $\Vb = \sbr{\vb_1,\cdots,\vb_{1000}}$ are $1000 \times 1000$ random sparse matrices with non-negative entries and sparsity $0.1$.
    \item {Let \texttt{Helmholtz} be a matrix $\Xb$ induced by the Helmholtz kernel:
    \begin{align*}
        G(\xb,\yb) = \exp\rbr{i \kappa \nbr{\xb - \yb}_2}/\rbr{4\pi \nbr{\xb - \yb}_2} \quad \t{with} \quad \kappa = 5.5.
    \end{align*}
    Consider a set of source points $\Scal$ consisting of $15 \times 15 \times 15$ Clenshaw-Curtis quadrature nodes in a 3D cube of size $2 \times 2 \times 2$ centered at the origin, along with a target set $\Tcal$ containing $2000$ uniformly distributed points on the surface of a 3D sphere with radius $3$, also centered at the origin. We construct $\Xb$ of size $3375 \times 2000$ by evaluating the Helmholtz kernel on $\Scal \times \Tcal$ such that $\Xb_{ij} = G(\xb_i,\yb_j)$ for every $\xb_i \in \Scal$ and $\yb_j \in \Tcal$.}
\end{enumerate}

{It is worth highlighting that ID on natural data like \texttt{MNIST}, \texttt{CIFAR-10}, and \texttt{Helmholtz} is closely related to practical applications like data selection and fast direct solvers.}
\begin{remark}[ID for data selection]\label{rmk:id_data_selection}
    {ID on a training dataset like \texttt{MNIST} and \texttt{CIFAR-10} can be viewed as an idealized data selection problem~\citep{alaoui2015fast,dong2024sketchy}. For example, consider data selection for an overparametrized ridge regression problem without noise: given a full dataset $(\Xb,\yb) \in \R^{n \times d} \times \R^n$ satisfying $\yb = \Xb\thetab_*$ for some unknown ground truth $\thetab_* \in \SSS^{d-1}$, we aim to approximate $\thetab_*$ via $\wh\thetab = \argmin_{\thetab \in \R^d} \|\Xb_S\thetab - \yb_S\|_2^2 + \alpha \|\thetab\|_2^2$ with some $\alpha > 0$ on a selected data subset $(\Xb_S,\yb_S) \in \R^{k \times d} \times \R^k$ of size $k = \abbr{S} \le d$. 
    In this setting, the skeletonization error of ID, $\errskel{\Xb}{S}$, is closely related to the generalization error of data selection: 
    \begin{align*}
        \t{as}\ \alpha \to 0, \quad \|\Xb(\wh\thetab - \thetab_*)\|_2^2 = \|(\Xb\Xb_S^\dagger \Xb_S - \Xb)\thetab_*\|_2^2 \le \errskel{\Xb}{S}.
    \end{align*}
    Beyond linear models like ridge regression, this connection can be extended to the more general finetuning problems for nonlinear models (like neural networks) by replacing $\Xb$ with the high-dimensional representations or model gradients evaluated at the data~\citep{dong2024sketchy}.}
\end{remark}

\begin{remark}[ID for fast direct solvers]\label{rmk:id_fast_direct_solvers}
    {In fast direct solvers for linear elliptic PDEs, ID is a widely used tool for compressing certain off-diagonal blocks of a rank-structured matrix~\citep{martinsson2005fast,gillman2012direct,ho2012fast,ho2013hierarchical,corona2015n,minden2017recursive}. In this setting, both a high-quality skeleton selection and an accurate interpolation matrix construction are crucial for attaining a nearly minimum compression error. The ID experiments on the \texttt{Helmholtz} matrix are designed to model a compression step in the state-of-the-art fast direct solvers for Lippmann-Schwinger equation.}
\end{remark}

\subsection{Accuracy and Efficiency of ID}\label{subsec:id_accuracy_efficiency}
In \Cref{fig:id_gmm,fig:id_mnist,fig:id_cifar10,fig:id_Gaussian_exp,fig:snn_a10,fig:helmholtz}, we compare accuracy (left) and efficiency (right) of different ID algorithms on data matrices in \Cref{subsec:data_matrices}.

\begin{figure}[!ht]
    \includegraphics[width=\textwidth]{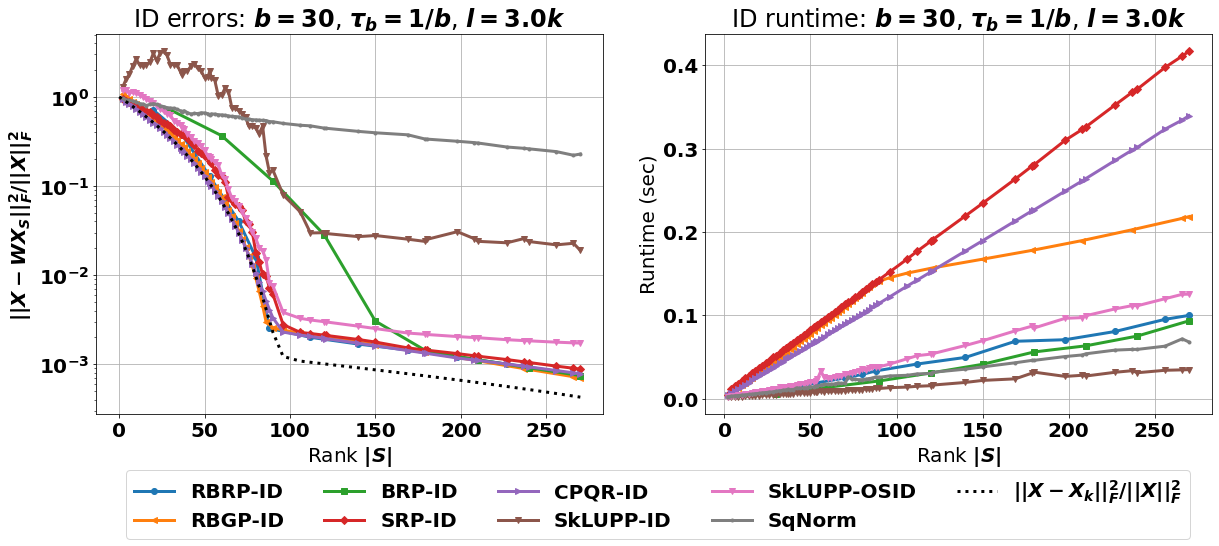}
    \caption{
        {Relative interpolation error and total runtime of ID algorithms on the \texttt{GMM} adversarial input (\Cref{ex:gmm_pitfall_plain_bas}). Recall from \Cref{fig:id_gmm_rpgp} that in lack of adaptiveness, methods like SqNorm and BRP tend to suffer from suboptimal skeleton complexities on \texttt{GMM}. }
    }\label{fig:id_gmm}
    \vspace{-1em}
\end{figure}

Taking an integrated view of ID including skeleton selection and interpolation matrix construction, we consider the following relative ID error\footnote{
    All error measurements are normalized with $\nbr{\Xb}_F^2$ to relative errors.
} and runtime measurements:
\begin{enumerate}[label=(\alph*)]
    \item For ID-revealing algorithms (\ie except for sketchy pivoting and sampling), we plot the runtimes and interpolation errors $\errid{\Xb}{S}{\Wb}$ with $\Wb$ from \Cref{algo:id} (\b{ID}), which are equal to the corresponding skeletonization errors $\errid{\Xb}{S}{\Wb} = \errskel{\Xb}{S}$.
    \item For sketchy pivoting, we plot the runtimes and interpolation errors $\errid{\Xb}{S}{\Wb}$ with $\Wb$ from both \Cref{algo:id} (\b{ID}) and \Cref{algo:fsid} (\b{OSID}) with small multiplicative oversampling $l = 3k$. Notice that both \b{ID} and \b{OSID} have $\errid{\Xb}{S}{\Wb} > \errskel{\Xb}{S}$, while \b{OSID} in \Cref{algo:fsid} tends to be more expensive because of oversampling.
    \item For squared-norm sampling, without explicitly specifying in the legends, we plot the runtimes and interpolation errors $\errid{\Xb}{S}{\Wb} = \errskel{\Xb}{S}$ with the optimal $\Wb$ from solving the exact least-square problem in \eqref{eq:stable_id}.
\end{enumerate}
The legends in \Cref{fig:id_gmm,fig:id_mnist,fig:id_cifar10,fig:id_Gaussian_exp,fig:snn_a10,fig:helmholtz} are formatted in terms of ``skeleton selection algorithm''-``ID error''. The abbreviations of different skeleton selection algorithms are summarized in \Cref{tab:abbrev_skeleton_selection}, whereas the ID error measurements are described above. 

\begin{figure}[!ht]
    \vspace{-1em}
    \includegraphics[width=\textwidth]{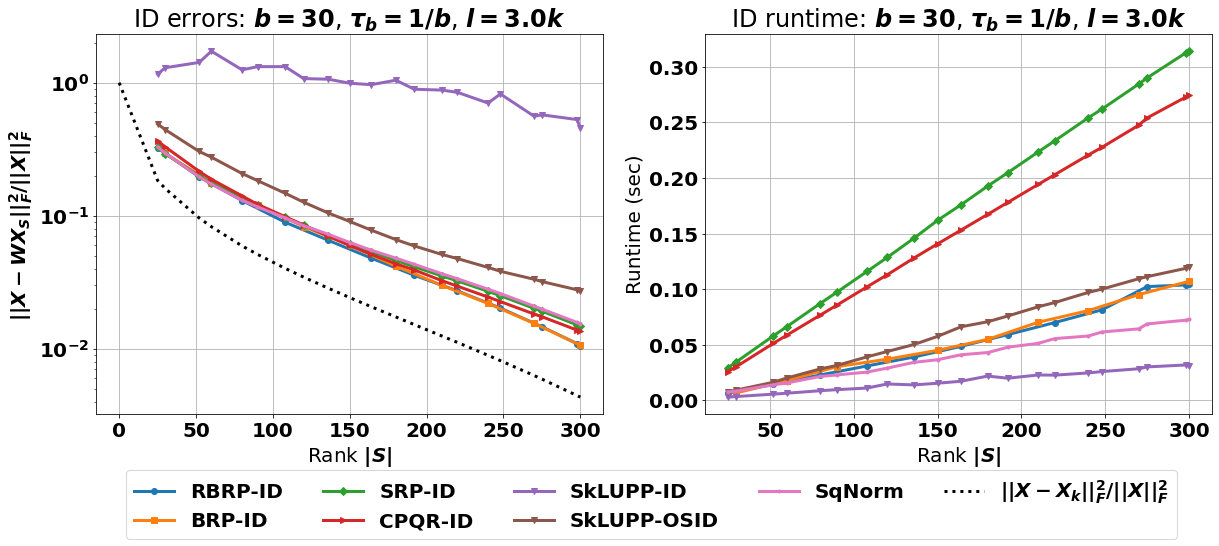}
    \caption{Relative interpolation error and runtime of ID algorithms on \texttt{MNIST}.}\label{fig:id_mnist}
    \vspace{-1em}
\end{figure}

\begin{figure}[!ht]
    \vspace{-1em}
    \includegraphics[width=\textwidth]{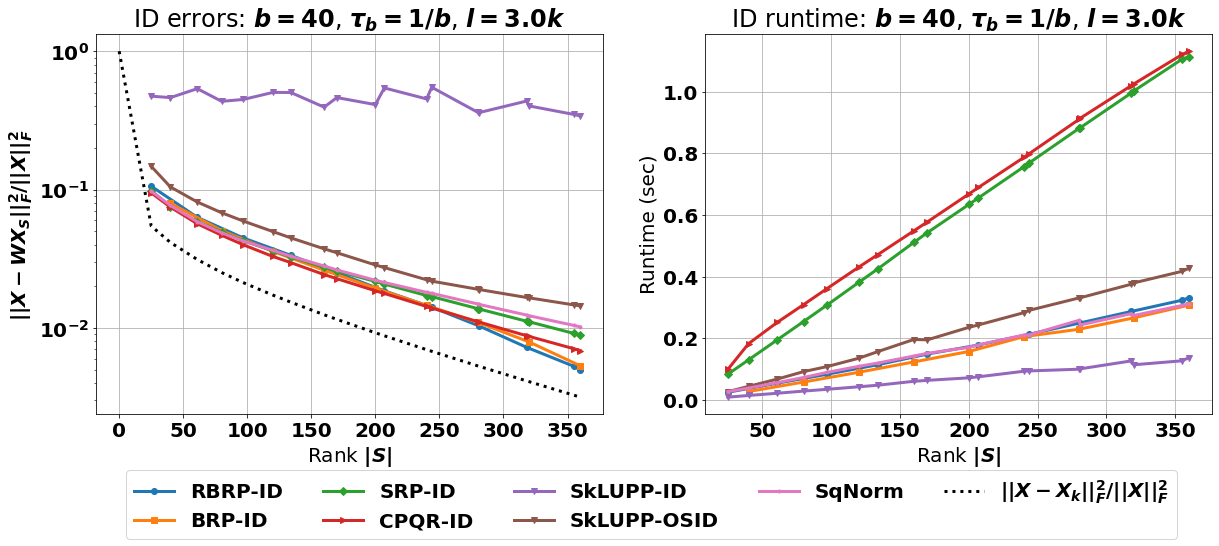}
    \caption{Relative interpolation error and runtime of ID algorithms on \texttt{CIFAR-10}.}\label{fig:id_cifar10}
    \vspace{-1em}
\end{figure}

\begin{figure}[!ht]
    \vspace{-1em}
    \includegraphics[width=\textwidth]{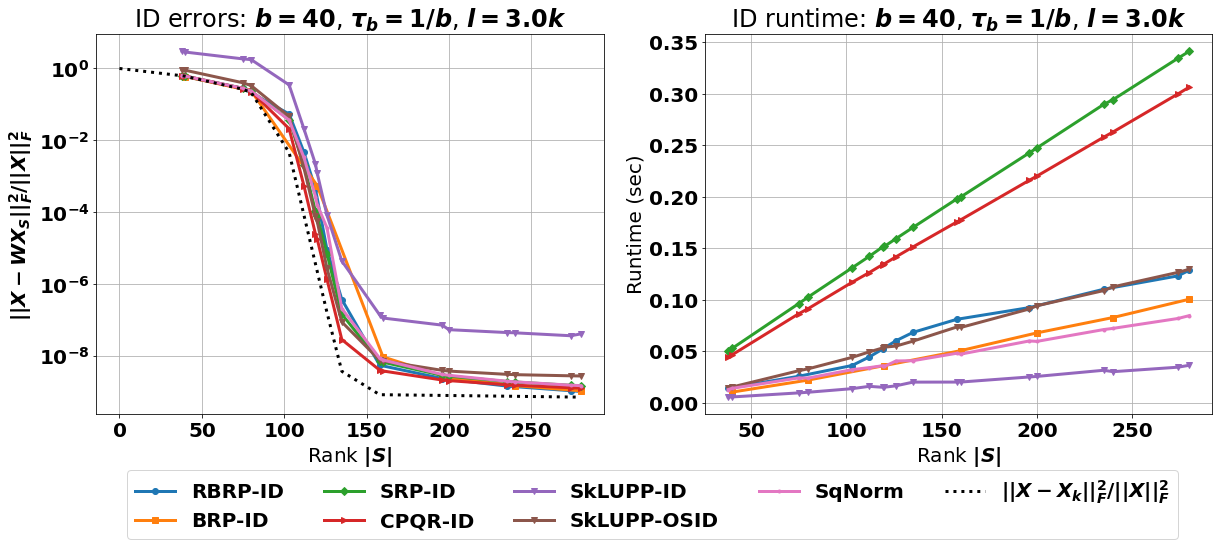}
    \caption{Relative interpolation error and runtime of ID algorithms on \texttt{Gaussian-exp}.}\label{fig:id_Gaussian_exp}
    \vspace{-1em}
\end{figure}

\begin{figure}[!ht]
    \vspace{-1em}
    \includegraphics[width=\textwidth]{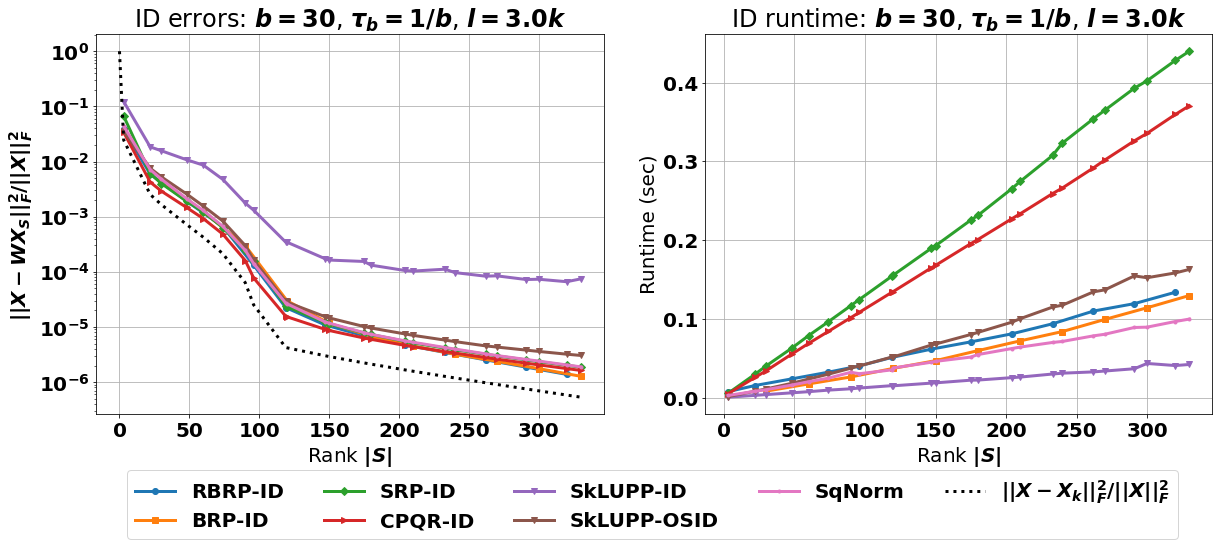}
    \caption{Relative interpolation error and runtime of ID algorithms on \texttt{SNN}.}\label{fig:snn_a10}
    \vspace{-1em}
\end{figure}

\begin{figure}[!ht]
    \vspace{-1em}
    \includegraphics[width=\textwidth]{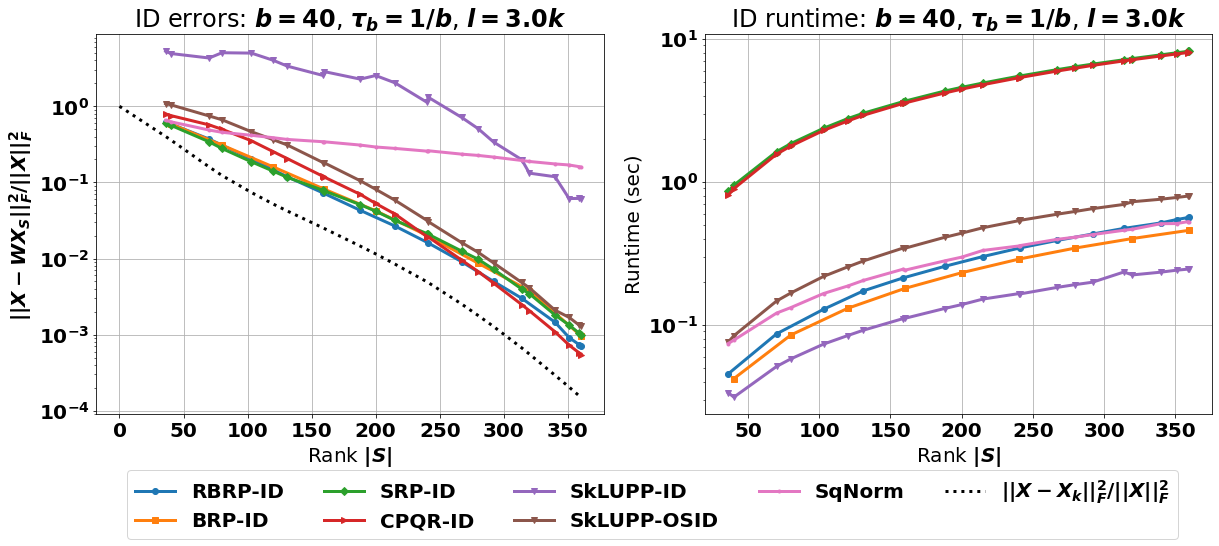}
    \caption{{Relative interpolation error and runtime of ID algorithms on \texttt{Helmholtz}.}}\label{fig:helmholtz}
    \vspace{-1em}
\end{figure}

\begin{remark}[Choice of block size]
    The choice of block size $b$ controls the trade-off between {hardware efficiency} and {rank-adaptiveness}. 
    To be precise, an unsuitably small $b$ lessens the efficiency gained from moving flops to matrix-matrix multiplications, whereas the larger $b$ leads to a higher overestimation of the target rank. In the experiments, we use $b=30$ or $40$ to exploit the {hardware efficiency} of matrix-matrix multiplications without considerably compromising the {rank-adaptiveness}.
\end{remark}

For the adversarial \texttt{GMM} dataset in \Cref{fig:id_gmm}, it is worth mentioning that the sharp transition of RBGP at $\abbr{S} \approx 100$ is because \texttt{GMM} contains $100$ clusters, each with distinct squared-norms. Recall that blockwise greedy pivoting picks points with the maximum norms without randomness. When $\abbr{S} \le 100$, each block drains points from the same clusters with the maximum residual norms first, most of which are redundant points and are later excluded by the robust blockwise filtering. This leads to the inefficiency of RBGP when $\abbr{S} \le 100$. 
Such inefficiency of RBGP in comparison to RBRP again illustrates the superiority of random pivoting over greedy pivoting, \ie the merit of randomness.

\begin{table}[!h]
    \vspace{-1em}
    \caption{\emph{GPU runtime} of the relatively efficient/accurate ID algorithms in \Cref{fig:id_gmm} on an adversarial \texttt{GMM} matrix (\Cref{ex:gmm_pitfall_plain_bas}) of size $10^5 \times 10^3$, with $k=100$ cluster centers. Experiments were performed on an Nvidia V100 GPU. Our code relied on the \texttt{magma\_dgeqp3\_gpu} routine from the MAGMA library ~\cite{dghklty14} for CPQR and the cuBLAS library~\cite{nickolls2008scalable} for other matrix operations.
    }\label{tab:gpu_runtime}
    \centering
    \begin{tabular}{cHcHcHcHcHc} \toprule
    %
    %
    rank  & 32	&   52	&   73	&  100	&  157	&  220	&  283	&  346	&  409	&  472	\\ \midrule
    CPQR-ID & 4.837e-02	&7.739e-02	&1.073e-01	&1.447e-01	&2.241e-01	&3.023e-01	&3.688e-01	&4.439e-01	&4.936e-01	& 5.611e-01 \\
    SkCPQR-OSID & 1.741e-02	&2.394e-02	&3.445e-02	&5.033e-02	&9.031e-02	&1.173e-01	&1.631e-01	&2.134e-01	&2.676e-01	&3.297e-01 \\
    SkLUPP-OSID & 1.036e-02 &1.204e-02 &1.666e-02 &2.233e-02 &3.430e-02 &4.835e-02 &6.960e-02 &8.254e-02 &1.017e-01 &1.189e-01   \\
    RBRP-ID & 3.105e-02	&4.448e-02	&5.838e-02	&7.231e-02	&9.512e-02	&1.073e-01	&1.217e-01	&1.336e-01	&1.473e-01	&1.596e-01 \\
    \bottomrule
    \end{tabular}
\end{table}

Further in \Cref{tab:gpu_runtime}, we compare the {GPU runtimes} of the {high-performance implementation for RBRP} on the \texttt{GMM} dataset with the other relatively efficient/accurate ID algorithms in \Cref{fig:id_gmm}, including
\begin{enumerate*}[label=(\roman*)]
    \item CPQR-ID with one of the best skeleton complexities (as RBRP-ID) and
    \item SkCPQR/SkLUPP-OSID with slightly compromised skeleton complexities (compared to RBRP-ID) but the best CPU runtimes.
\end{enumerate*}
We observe that RBRP-ID is consistently more efficient than CPQR-ID across all ranks. Moreover, as the rank increases, the efficiency of RBRP-ID approaches that of SkCPQR/SkLUPP-OSID, in addition to better accuracy.

From the (non-adversarial) datasets in \Cref{fig:id_mnist,fig:id_cifar10,fig:id_Gaussian_exp,fig:snn_a10,fig:helmholtz}, we observe the following:
\begin{itemize}
    \item \b{RBRP} enjoys (one of) the best skeleton complexities.
    In terms of efficiency, RBRP tends to be slower than sketchy pivoting (SkLUPP) with ID (which is neither {rank-adaptive} nor ID-revealing). However, RBRP is usually faster than sketchy pivoting with OSID, while enjoying lower interpolation error. Overall, RBRP is shown to provide the best combination of accuracy and efficiency when the explicit construction of ID (in contrast to the skeleton subset only) is inquired.
    \item \b{Sketchy pivoting (SkLUPP)} also enjoy (one of) the best skeleton complexities, despite the lack of {rank-adaptiveness} and ID-revealing property (\ie suffering from large interpolation errors). That is, for skeleton selection only (without asking for the interpolation matrix), sketchy pivoting provides the most efficient selection of close-to-the-best skeleton subsets in practice (\cf \Cref{fig:id_gmm_rpgp}). 
    \item The sequential algorithms \b{SRP} and \b{CPQR} are highly competitive in terms of accuracy while being considerably slower than the {hardware-efficient blockwise algorithms}.
    \item With approximately normalized data in \Cref{fig:id_mnist,fig:id_cifar10,fig:id_Gaussian_exp,fig:snn_a10,fig:helmholtz}, {plain \b{BRP} and} \b{SqNorm}\footnote{
        Although we follow \cite{deshpande2006matrix} and show results by sampling $k~\iid$ skeletons with \emph{replacement}, in practice, we observe that sampling \emph{without replacement} tends to provide better skeleton complexity, especially on the adversarial input \Cref{ex:gmm_pitfall_plain_bas}. Intuitively, sampling {without replacement} can be viewed as a weaker version of adaptiveness where only the selected point itself is excluded from the future selection.
    } also provide competitive skeleton complexities. {In these benign cases, BRP can enjoy slightly better runtime than RBRP. However, adversarial inputs like \Cref{ex:gmm_pitfall_plain_bas} can be easily constructed to break such benign cases (\cf \Cref{fig:id_gmm}). RBRP offers an effective remedy for the vulnerability of plain BRP with negligible additional cost.}
    \item When using \b{SqNorm} as an ID algorithm, constructing the interpolation matrix by solving \eqref{eq:stable_id} explicitly can slow down the overall ID algorithm significantly to the extent that the runtimes are similar to those of RBRP. Moreover, sampling methods like SqNorm and DPP lack {rank-adaptiveness} and ID-revealing ability.
\end{itemize}

\section{Discussion}
In this work, we focus on fast and accurate ID algorithms from five perspectives that measure the empirical performance systematically: 
\begin{enumerate*}[label=(\alph*)]
    \item skeleton complexity,
    \item asymptotic complexity,
    \item {hardware efficiency},
    \item {rank-adaptiveness}, and 
    \item ID-revealing property.
\end{enumerate*}
With a careful exploration of various existing ID algorithms through the lens of adaptiveness and/or randomness, we reemphasize the effectiveness of combining adaptiveness and randomness in solving ID problems, while unveiling a critical missing piece in the family of existing algorithms that achieves all five of the aforementioned desired properties. 
To close the gap, we propose \emph{robust blockwise random pivoting (RBRP)}---a {hardware-efficient}, {rank-adaptive}, and ID-revealing algorithm that enjoys among-the-best skeleton and asymptotic complexities in practice compared to the existing algorithms and demonstrates promising robustness to adversarial inputs.
Through numerical experiments, we illustrate the empirical accuracy, efficiency, and robustness of RBRP on a broad spectrum of natural and synthetic datasets.


\section*{Acknowledgments}
The authors wish to thank Yifan Chen, Ethan Epperly, Kevin Miller, Christopher Musco, Rachel Ward, and Robert Webber, as well as the anonymous referees and editors, for enlightening discussions and constructive feedback. 


\bibliographystyle{siamplain}
\bibliography{ref}

\begin{thebibliography}{10}

\bibitem{alaoui2015fast}
{\sc A.~Alaoui and M.~W. Mahoney}, {\em Fast randomized kernel ridge regression
  with statistical guarantees}, Advances in neural information processing
  systems, 28 (2015).

\bibitem{arthur2007k}
{\sc D.~Arthur and S.~Vassilvitskii}, {\em K-means++ the advantages of careful
  seeding}, in Proceedings of the eighteenth annual ACM-SIAM symposium on
  Discrete algorithms, 2007, pp.~1027--1035.

\bibitem{avron2013faster}
{\sc H.~Avron and C.~Boutsidis}, {\em Faster subset selection for matrices and
  applications}, SIAM Journal on Matrix Analysis and Applications, 34 (2013),
  pp.~1464--1499.

\bibitem{axiotis2024data}
{\sc K.~Axiotis, V.~Cohen-Addad, M.~Henzinger, S.~Jerome, V.~Mirrokni,
  D.~Saulpic, D.~Woodruff, and M.~Wunder}, {\em Data-efficient learning via
  clustering-based sensitivity sampling: Foundation models and beyond}, arXiv
  preprint arXiv:2402.17327,  (2024).

\bibitem{belabbas2009spectral}
{\sc M.-A. Belabbas and P.~J. Wolfe}, {\em Spectral methods in machine learning
  and new strategies for very large datasets}, Proceedings of the National
  Academy of Sciences, 106 (2009), pp.~369--374.

\bibitem{hough2006determinantal}
{\sc J.~Ben~Hough, M.~Krishnapur, Y.~Peres, and B.~Vir{\'a}g}, {\em
  Determinantal processes and independence.}, Probability Surveys [electronic
  only], 3 (2006), pp.~206--229.

\bibitem{blackford2002updated}
{\sc L.~S. Blackford, A.~Petitet, R.~Pozo, K.~Remington, R.~C. Whaley,
  J.~Demmel, J.~Dongarra, I.~Duff, S.~Hammarling, G.~Henry, et~al.}, {\em {An
  updated set of basic linear algebra subprograms (BLAS)}}, ACM Transactions on
  Mathematical Software, 28 (2002), pp.~135--151.

\bibitem{chen2022randomly}
{\sc Y.~Chen, E.~N. Epperly, J.~A. Tropp, and R.~J. Webber}, {\em {Randomly
  pivoted Cholesky: Practical approximation of a kernel matrix with few entry
  evaluations}}, arXiv preprint arXiv:2207.06503,  (2022).

\bibitem{chenakkod2024optimal}
{\sc S.~Chenakkod, M.~Derezi{\'n}ski, X.~Dong, and M.~Rudelson}, {\em Optimal
  embedding dimension for sparse subspace embeddings}, in Proceedings of the
  56th Annual ACM Symposium on Theory of Computing, 2024, pp.~1106--1117.

\bibitem{2005_martinsson_skel}
{\sc H.~Cheng, Z.~Gimbutas, P.~Martinsson, and V.~Rokhlin}, {\em On the
  compression of low rank matrices}, SIAM Journal of Scientific Computing, 26
  (2005), pp.~1389--1404.

\bibitem{cohen2016nearly}
{\sc M.~B. Cohen}, {\em Nearly tight oblivious subspace embeddings by trace
  inequalities}, in Proceedings of the twenty-seventh annual ACM-SIAM symposium
  on Discrete algorithms, SIAM, 2016, pp.~278--287.

\bibitem{cohen2015uniform}
{\sc M.~B. Cohen, Y.~T. Lee, C.~Musco, C.~Musco, R.~Peng, and A.~Sidford}, {\em
  Uniform sampling for matrix approximation}, in Proceedings of the 2015
  Conference on Innovations in Theoretical Computer Science, 2015,
  pp.~181--190.

\bibitem{corona2015n}
{\sc E.~Corona, P.-G. Martinsson, and D.~Zorin}, {\em An o (n) direct solver
  for integral equations on the plane}, Applied and Computational Harmonic
  Analysis, 38 (2015), pp.~284--317.

\bibitem{cortinovis2020low}
{\sc A.~Cortinovis and D.~Kressner}, {\em {Low-rank approximation in the
  Frobenius norm by column and row subset selection}}, SIAM Journal on Matrix
  Analysis and Applications, 41 (2020), pp.~1651--1673.

\bibitem{dasgupta2008hierarchical}
{\sc S.~Dasgupta and D.~Hsu}, {\em Hierarchical sampling for active learning},
  in Proceedings of the 25th international conference on Machine learning,
  2008, pp.~208--215.

\bibitem{derezinski2019fast}
{\sc M.~Derezi{\'n}ski}, {\em Fast determinantal point processes via
  distortion-free intermediate sampling}, in Conference on Learning Theory,
  PMLR, 2019, pp.~1029--1049.

\bibitem{derezinski2020improved}
{\sc M.~Derezinski, R.~Khanna, and M.~W. Mahoney}, {\em {Improved guarantees
  and a multiple-descent curve for Column Subset Selection and the Nystrom
  method}}, Advances in Neural Information Processing Systems, 33 (2020),
  pp.~4953--4964.

\bibitem{derezinski2021sparse}
{\sc M.~Derezinski, Z.~Liao, E.~Dobriban, and M.~Mahoney}, {\em Sparse sketches
  with small inversion bias}, in Conference on Learning Theory, PMLR, 2021,
  pp.~1467--1510.

\bibitem{derezinski2021determinantal}
{\sc M.~Derezinski and M.~W. Mahoney}, {\em Determinantal point processes in
  randomized numerical linear algebra}, Notices of the American Mathematical
  Society, 68 (2021), pp.~34--45.

\bibitem{deshpande2006matrix}
{\sc A.~Deshpande, L.~Rademacher, S.~S. Vempala, and G.~Wang}, {\em Matrix
  approximation and projective clustering via volume sampling}, Theory of
  Computing, 2 (2006), pp.~225--247.

\bibitem{deshpande2006adaptive}
{\sc A.~Deshpande and S.~Vempala}, {\em Adaptive sampling and fast low-rank
  matrix approximation}, in International Workshop on Approximation Algorithms
  for Combinatorial Optimization, Springer, 2006, pp.~292--303.

\bibitem{dobriban2019asymptotics}
{\sc E.~Dobriban and S.~Liu}, {\em Asymptotics for sketching in least squares
  regression}, Advances in Neural Information Processing Systems, 32 (2019).

\bibitem{dong2021simpler}
{\sc Y.~Dong and P.-G. Martinsson}, {\em Simpler is better: a comparative study
  of randomized pivoting algorithms for cur and interpolative decompositions},
  Advances in Computational Mathematics, 49 (2023), p.~66.

\bibitem{dong2024randomly}
{\sc Y.~Dong, X.~Pan, H.~Phan, and Q.~Lei}, {\em Randomly pivoted v-optimal
  design: Fast data selection under low intrinsic dimension}, in Workshop on
  Machine Learning and Compression, NeurIPS 2024.

\bibitem{dong2024sketchy}
{\sc Y.~Dong, H.~Phan, X.~Pan, and Q.~Lei}, {\em Sketchy moment matching:
  Toward fast and provable data selection for finetuning}, in The Thirty-eighth
  Annual Conference on Neural Information Processing Systems, 2024.

\bibitem{dghklty14}
{\sc J.~Dongarra, M.~Gates, A.~Haidar, J.~Kurzak, P.~Luszczek, S.~Tomov, and
  I.~Yamazaki}, {\em {Accelerating Numerical Dense Linear Algebra Calculations
  with GPUs}}, Numerical Computations with GPUs,  (2014), pp.~1--26.

\bibitem{drineas2012fast}
{\sc P.~Drineas, M.~Magdon-Ismail, M.~W. Mahoney, and D.~P. Woodruff}, {\em
  Fast approximation of matrix coherence and statistical leverage}, The Journal
  of Machine Learning Research, 13 (2012), pp.~3475--3506.

\bibitem{2017_blockQR_ming_article}
{\sc J.~A. Duersch and M.~Gu}, {\em {Randomized QR with Column Pivoting}}, SIAM
  Journal on Scientific Computing, 39 (2017), pp.~C263--C291.

\bibitem{epperly2023kernel}
{\sc E.~Epperly and E.~Moreno}, {\em Kernel quadrature with randomly pivoted
  cholesky}, Advances in Neural Information Processing Systems, 36 (2023),
  pp.~65850--65868.

\bibitem{frieze2004fast}
{\sc A.~Frieze, R.~Kannan, and S.~Vempala}, {\em {Fast Monte-Carlo algorithms
  for finding low-rank approximations}}, Journal of the ACM (JACM), 51 (2004),
  pp.~1025--1041.

\bibitem{GPBV19}
{\sc G.~Gautier, G.~Polito, R.~Bardenet, and M.~Valko}, {\em {DPPy: DPP
  Sampling with Python}}, Journal of Machine Learning Research - Machine
  Learning Open Source Software (JMLR-MLOSS),  (2019).
\newblock Code at http://github.com/guilgautier/DPPy/ Documentation at
  http://dppy.readthedocs.io/.

\bibitem{gillman2012direct}
{\sc A.~Gillman, P.~M. Young, and P.-G. Martinsson}, {\em A direct solver with
  o (n) complexity for integral equations on one-dimensional domains},
  Frontiers of Mathematics in China, 7 (2012), pp.~217--247.

\bibitem{golub2013matrix}
{\sc G.~H. Golub and C.~F. Van~Loan}, {\em Matrix computations}, JHU press,
  2013.

\bibitem{goto2008high}
{\sc K.~Goto and R.~Van De~Geijn}, {\em {High-performance implementation of the
  level-3 BLAS}}, ACM Transactions on Mathematical Software (TOMS), 35 (2008),
  pp.~1--14.

\bibitem{gu1996efficient}
{\sc M.~Gu and S.~C. Eisenstat}, {\em Efficient algorithms for computing a
  strong rank-revealing qr factorization}, SIAM Journal on Scientific
  Computing, 17 (1996), pp.~848--869.

\bibitem{guruswami2012optimal}
{\sc V.~Guruswami and A.~K. Sinop}, {\em Optimal column-based low-rank matrix
  reconstruction}, in Proceedings of the twenty-third annual ACM-SIAM symposium
  on Discrete Algorithms, SIAM, 2012, pp.~1207--1214.

\bibitem{halko2011finding}
{\sc N.~Halko, P.-G. Martinsson, and J.~A. Tropp}, {\em {Finding structure with
  randomness: Probabilistic algorithms for constructing approximate matrix
  decompositions}}, SIAM review, 53 (2011), pp.~217--288.

\bibitem{higham1990analysis}
{\sc N.~J. Higham}, {\em Analysis of the cholesky decomposition of a
  semi-definite matrix},  (1990).

\bibitem{ho2012fast}
{\sc K.~L. Ho and L.~Greengard}, {\em A fast direct solver for structured
  linear systems by recursive skeletonization}, SIAM Journal on Scientific
  Computing, 34 (2012), pp.~A2507--A2532.

\bibitem{ho2013hierarchical}
{\sc K.~L. Ho and L.~Ying}, {\em Hierarchical interpolative factorization for
  elliptic operators: Integral equations}, Communications on Pure and Applied
  Mathematics, 7 (2016), pp.~1314--1353.

\bibitem{houlsby2019parameter}
{\sc N.~Houlsby, A.~Giurgiu, S.~Jastrzebski, B.~Morrone, Q.~De~Laroussilhe,
  A.~Gesmundo, M.~Attariyan, and S.~Gelly}, {\em Parameter-efficient transfer
  learning for nlp}, in International conference on machine learning, PMLR,
  2019, pp.~2790--2799.

\bibitem{hu2021lora}
{\sc E.~J. Hu, yelong shen, P.~Wallis, Z.~Allen-Zhu, Y.~Li, S.~Wang, L.~Wang,
  and W.~Chen}, {\em Lo{RA}: Low-rank adaptation of large language models}, in
  International Conference on Learning Representations, 2022.

\bibitem{kahan1966numerical}
{\sc W.~Kahan}, {\em Numerical linear algebra}, Canadian Mathematical Bulletin,
  9 (1966), pp.~757--801.

\bibitem{krizhevsky2009learning}
{\sc A.~Krizhevsky, G.~Hinton, et~al.}, {\em Learning multiple layers of
  features from tiny images},  (2009).

\bibitem{kulesza2011k}
{\sc A.~Kulesza and B.~Taskar}, {\em k-dpps: Fixed-size determinantal point
  processes}, in Proceedings of the 28th International Conference on Machine
  Learning (ICML-11), 2011, pp.~1193--1200.

\bibitem{lecun1998mnist}
{\sc Y.~LeCun}, {\em {The MNIST database of handwritten digits}}, http://yann.
  lecun. com/exdb/mnist/,  (1998).

\bibitem{lester2021power}
{\sc B.~Lester, R.~Al-Rfou, and N.~Constant}, {\em The power of scale for
  parameter-efficient prompt tuning}, arXiv preprint arXiv:2104.08691,  (2021).

\bibitem{li2024greedy}
{\sc J.~Li, Y.~Dong, and Q.~Lei}, {\em Greedy output approximation: Towards
  efficient structured pruning for llms without retraining}, arXiv preprint
  arXiv:2407.19126,  (2024).

\bibitem{mahoney2009cur}
{\sc M.~W. Mahoney and P.~Drineas}, {\em {CUR matrix decompositions for
  improved data analysis}}, Proceedings of the National Academy of Sciences,
  106 (2009), pp.~697--702.

\bibitem{martinsson2017householder}
{\sc P.-G. Martinsson, G.~Quintana~OrtI, N.~Heavner, and R.~Van De~Geijn}, {\em
  {Householder QR factorization with randomization for column pivoting
  (HQRRP)}}, SIAM Journal on Scientific Computing, 39 (2017), pp.~C96--C115.

\bibitem{martinsson2005fast}
{\sc P.-G. Martinsson and V.~Rokhlin}, {\em A fast direct solver for boundary
  integral equations in two dimensions}, Journal of Computational Physics, 205
  (2005), pp.~1--23.

\bibitem{martinsson2020randomized}
{\sc P.-G. Martinsson and J.~A. Tropp}, {\em Randomized numerical linear
  algebra: Foundations and algorithms}, Acta Numerica, 29 (2020), pp.~403--572.

\bibitem{minden2017recursive}
{\sc V.~Minden, K.~L. Ho, A.~Damle, and L.~Ying}, {\em A recursive
  skeletonization factorization based on strong admissibility}, Multiscale
  Modeling \& Simulation, 15 (2017), pp.~768--796.

\bibitem{musco2017sublinear}
{\sc C.~Musco and D.~P. Woodruff}, {\em Sublinear time low-rank approximation
  of positive semidefinite matrices}, in 2017 IEEE 58th Annual Symposium on
  Foundations of Computer Science (FOCS), IEEE, 2017, pp.~672--683.

\bibitem{nickolls2008scalable}
{\sc J.~Nickolls, I.~Buck, M.~Garland, and K.~Skadron}, {\em Scalable parallel
  programming with cuda: Is cuda the parallel programming model that
  application developers have been waiting for?}, Queue, 6 (2008), pp.~40--53.

\bibitem{pilanci2016iterative}
{\sc M.~Pilanci and M.~J. Wainwright}, {\em {Iterative Hessian sketch: Fast and
  accurate solution approximation for constrained least-squares}}, The Journal
  of Machine Learning Research, 17 (2016), pp.~1842--1879.

\bibitem{pukelsheim2006optimal}
{\sc F.~Pukelsheim}, {\em Optimal design of experiments}, SIAM, 2006.

\bibitem{quintana1998blas}
{\sc G.~Quintana-Ort{\'\i}, X.~Sun, and C.~H. Bischof}, {\em A blas-3 version
  of the qr factorization with column pivoting}, SIAM Journal on Scientific
  Computing, 19 (1998), pp.~1486--1494.

\bibitem{raskutti2016statistical}
{\sc G.~Raskutti and M.~W. Mahoney}, {\em A statistical perspective on
  randomized sketching for ordinary least-squares}, The Journal of Machine
  Learning Research, 17 (2016), pp.~7508--7538.

\bibitem{sorensen2016deim}
{\sc D.~C. Sorensen and M.~Embree}, {\em {A DEIM-induced CUR factorization}},
  SIAM Journal on Scientific Computing, 38 (2016), pp.~A1454--A1482.

\bibitem{trefethen1990average}
{\sc L.~N. Trefethen and R.~S. Schreiber}, {\em {Average-case stability of
  Gaussian elimination}}, SIAM Journal on Matrix Analysis and Applications, 11
  (1990), pp.~335--360.

\bibitem{voronin2017efficient}
{\sc S.~Voronin and P.-G. Martinsson}, {\em {Efficient algorithms for CUR and
  interpolative matrix decompositions}}, Advances in Computational Mathematics,
  43 (2017), pp.~495--516.

\bibitem{wang2024pmss}
{\sc Q.~Wang, X.~Hu, W.~Xu, W.~Liu, J.~Luan, and B.~Wang}, {\em Pmss:
  Pretrained matrices skeleton selection for llm fine-tuning}, arXiv preprint
  arXiv:2409.16722,  (2024).

\bibitem{woodruff2014sketching}
{\sc D.~P. Woodruff et~al.}, {\em Sketching as a tool for numerical linear
  algebra}, Foundations and Trends{\textregistered} in Theoretical Computer
  Science, 10 (2014), pp.~1--157.

\end{thebibliography}

\end{document}